\documentclass[12pt]{article}
\usepackage{graphicx}
\usepackage{color}
\usepackage{setspace}
\usepackage{eufrak}
\usepackage{float}
\usepackage{etex}
\usepackage[all]{xy}
\usepackage{amssymb, amsmath}
\usepackage{ascmac}
\usepackage{hyperref}
\hypersetup{colorlinks=true}

\def\qed{\hfill $\Box$}

\newcommand{\R}{\mathbb{R}}

\newcommand\rank{\mbox{\rm rank}\,}

\def\qed{\hfill $\Box$}

\def\bx{\mbox{\boldmath $x$}}

\def\be{\mbox{\boldmath $e$}}

\def\b0{\mbox{\boldmath $0$}}

\newtheorem{thm}{\bf Theorem}[section]
\newtheorem{cor}[thm]{\bf Corollary}

\newtheorem{prop}[thm]{\bf Proposition}

\newtheorem{rem}[thm]{\bf Remark}

\begin{document}
\title{On the differential geometry of smooth ruled surfaces in $4$-space}
\author{Jorge Luiz Deolindo-Silva}

\date{ }

%\author[J. L. ~Deolindo Silva]{Jorge Luiz Deolindo Silva}
%\address[J. L.~Deolindo-Silva]{{ Departamento de Matem\'atica,
%Universidade Federal de Santa Catarina-Blumenau-SC 89036-004, Brazil}}
%\email{jorge.deolindo@ufsc.br}
%
%\author[Y.~Kabata]{Yutaro Kabata}
%\address[Y. ~Kabata]{Department of Mathematics,
%Graduate School of Science,  Hokkaido University,
%Sapporo 060-0810, Japan}
%\email{kabata@mail.sci.hokudai.ac.jp}
%
%\author[T.~Ohmoto]{Toru Ohmoto}
%\address[T.~Ohmoto]{Department of Mathematics,
%Faculty of Science,  Hokkaido University,
%Sapporo 060-0810, Japan}
%\email{ohmoto@math.sci.hokudai.ac.jp}
%
%\subjclass[2010]{58K05, 32S15, 34A09, 53A20}
%
%\keywords{singularities of smooth maps, ruled surfaces, flecnodal curves, projection of surfaces,  projective differential geometry. }
%
%\dedicatory{}
%\date{\today}
%
\maketitle

\begin{abstract}
%A smooth ruled surface in 4-space has only parabolic points or inflection points of real type. Through contact with transverse planes, we show that at each point in parabolic region of a ruled surface, there are two tangent directions determining two planes along which the projection has $\mathcal A$-singularities of type butterfly or worse. In particular, such points can be classified as butterfly hyperbolic, parabolic, or elliptic points depending on the value of the discriminant of a binary differential equation (BDE). We ensure that the integral curves of these directions form a pair of foliations on the ruled surface when the discriminant is positive and the set of points that nullify the discriminant of this BDE is a regular curve transverse to the regular curve formed by inflection points of real type. 
%In addition, by projective transformation we obtain, using coefficients of a simple normal form, the stable configuration of the solutions of BDE in the discriminant curve.

A smooth ruled surface in 4-space has only parabolic points or inflection points of real type.  We show, by means of contact with transverse planes, that at a parabolic point, there exist two tangent directions determining two planes along which the parallel projection exhibit $\mathcal A$-singularities of type butterfly or worse.  In particular, such parabolic point can be classified as butterfly hyperbolic, parabolic, or elliptic point depending on the value of the discriminant of a binary differential equation (BDE). 
Also,  whenever such discriminant is positive, we ensure that the integral curves of these directions form a pair of foliations on the ruled surface.  Moreover, the set of points that nullify the discriminant is a regular curve transverse to the regular curve formed by inflection points of real type. 
Finally, using a particular projective transformation, we obtain  a simple parametrization of the ruled surface such that the moduli of its 5-jet identify a butterfly hyperbolic/parabolic/elliptic point, as well as we get the stable configurations of the solutions of BDE in the discriminant curve.
\end{abstract}

\renewcommand{\thefootnote}{\fnsymbol{footnote}}

\setlength{\baselineskip}{15pt}

%---------------------------------------------------------
\section{Introduction}

%The study of ruled surfaces is a classical subject in differential geometry. 

%Ruled surfaces may be considered as a one-parameter family of lines and its study is one of the most important topics in classical differential geometry. They can play an important role in various areas such as computation, kinematics, and architecture (\cite{Altin,PW,IFRT-book}). Ruled surfaces in 4-space have been studied using several approaches, see for example \cite{Bayram, Plass, OS, Ranum, Roschel}. This paper aims to investigate the differential geometry of a smooth ruled surface in 4-space (Euclidean, affine, or projective) using singularity theory techniques.  

%The geometry of smooth surfaces in 4-space is extensively studied, including from a singularity theory viewpoint (see for example \cite[Chapter 7]{IFRT-book} for references). The contact of a surface in 4-space with lines, planes, and hyperplanes captures some aspect of the geometry of the surface. 

Ruled surfaces may be considered as a one-parameter family of lines (the {\it rulings}) and its study is one of the most important topics in classical differential geometry. They play a relevant role in various areas such as computation, kinematics, and architecture (\cite{Altin,PW,IFRT-book}). Ruled surfaces in 4-space have been extensively studied using several approaches, see for example \cite{Bayram, Plass, OS, Ranum, Roschel}. This paper aims to investigate the differential geometry of a smooth (regular) ruled surface in 4-space (Euclidean, affine, or projective) from a singularity theory viewpoint. 

Singularity theory has proved to be an important framework for studying smooth and regular surfaces in 4-space. The contact of a surface in 4-space with flat objects (hyperplanes, planes, and lines) captures some aspect of the geometry of the surface (see the recent book \cite[Chapter 7]{IFRT-book} for references). 
For instance, the contact with lines and hyperplanes gives robust information at hyperbolic and parabolic points. %(see Section \ref{pre}). 

For a smooth ruled surface in 4-space, every point is a parabolic point (\cite{Bayram}) or an inflection point of real type (see Proposition \ref{class}). More precisely, for each ruling, all points are parabolic except one which is an inflection point of real type. Then when making the contact between the ruled surface with lines and hyperplanes, the contact is highly degenerate (because the ruling is itself an asymptotic direction) and does not exhibit any significant geometric characteristics on ruled surface. For this reason, we consider only the contact of ruled surfaces with planes. Such contact is measured by the $\mathcal A$-singularities of parallel projections of the ruled surface along planes to transverse planes (see \cite{bruce-nogueira,DT}). (Two map-germs $f$ and $g$ are said to be $\mathcal A$-equivalent if
$g = k \circ f \circ h^{-1}$ for some germs of diffeomorphisms $h$ and $k$ of, respectively, the
source and target.)

At a parabolic point of a ruled surface, we show that there are two tangent directions which determine two planes in $\mathbb R^4$ along which the parallel projection has a butterfly singularity. These directions are given by a binary differential equation (BDE). For a generic ruled surface, the set of the inflection points of real type and the discriminant curve of this BDE (namely by {\it inflection curve} and {\it butterfly parabolic curve}, respectively) are regular transverse curves. Moreover,  in an analogy to the behavior of the asymptotic directions on surfaces in $4$-space, a parabolic point  
can be reclassified as {\it butterfly hyperbolic/parabolic/elliptic}, depending on whether the discriminant of the BDE is positive/null/negative, respectively.

As the contact of a ruled surface with planes is also projective invariant, then we obtain in Section \ref{SectProj} that the 4-jet of a parametrization of a ruled surface at a parabolic point can be taken, via a projective transformation, in the form
$(x,y)\mapsto(x,y,f^1(x,y), f^2(x,y))$ where $f^1(x,y)=x^2+\Gamma_{40} x^4+\Gamma_{31} x^3y$ and $f^2(x,y)=xy+\Theta_{40} x^4$. In Section \ref{secBDE}, %using this normal form, 
we obtain that the butterfly hyperbolic/parabolic/elliptic point can be identified when $\Gamma_{40}^2+\Gamma_{31}\Theta_{40}>0/=0/<0$, respectively. Finally, we give the stable configurations of the solutions of BDE in the butterfly parabolic curve using the moduli of the normal form of Corollary \ref{5-jet}.

%In 1936 Plass studied ruled surfaces imbedded in a Euclidean space of 4-space. %Curvature properties of the surface are investigated with respect to the variation of normal vectors and a curvature conic along a generator of the surface \cite{}. 
%and a theory of ruled surface in Euclidean 4-space was developed by \cite{T. Otsuiki and K. Shiohama}.

%For example in \cite{izumyia}, is a survey article with original results on the recent development of the study of the regular and singular ruled surfaces in Euclidean space using the singularity technics in this object.

\section{Preliminary}\label{pre}

\subsection{Ruled surface in $4$-space}
We recall some basic notions  \cite{PW, Plass}. As we are interested in local geometry of a ruled surface,  we work in an affine chart  $\R^4=\{[x:y:z:w:1]\} \subset \mathbb RP^4$. Thus  take a parametrization map  $F: I \times \R \to \R^4$ given by
\begin{equation}\label{ruled}
F(u,t):=\bx(u)+t \be(u)
\end{equation}
where $I$ is an open real interval with $\be(u)\not=\vec{0}$, and  
the curves $\bx(u)$ and $\be(u)$ are called the {\it base or directrix} curve and the {\it director curve}, respectively. 

Let $R$ be the image surface, and let $R_u$ be the {\it ruling} corresponding to the parameter $u \in I$. We say that the ruling $R_{u_0}$ is {\it singular} if $\bx'(u_0)=\be'(u_0)=\vec{0}$ and {\it regular} otherwise. 

As usual, a {\it smooth point} $p=(u_0,t_0)$ at a ruled surface refers to the point where  
$F$ is immersive, i.e., $F_u(p)=\bx'(u_0)+t_0\be'(u_0)$ and $F_t(p)=\be(u_0)$ are linearly independent. Otherwise, it is termed a {\it singular point}. 
Note that any point on a singular ruling is a singular point on the surface. A regular ruling has, at most, one singular point at the surface. Throughout this paper, our focus will only be on the smooth points of a ruled surface.

%---------------------------------------------------

\subsection{Surfaces in 4-space}\label{subs:surfr4}
Let $M$ be a smooth (i.e., of class $C^{\infty}$) and regular surface in $\mathbb R^4$ (i.e., an embedding). Locally at $p$, we write $M$ as the image of some smooth map $\varphi(x,y)=(x,y,f^1(x,y),f^2(x,y))$, where $(x,y)$ is in an open neighbourhood $U$ of $\mathbb R^2$ containing the origin. We take $\varphi(0,0)=p$ and $j^1f^i(0,0)=0$, $i=1,2$. This kind of parametrisation is called {\it Monge form}. Following the approach in \cite{bruce-nogueira},  we have  a pair of quadratic forms
$$
(Q_1,Q_2)=(lx^2+2mxy+ny^2,ax^2+2bxy+cy^2)
$$ 
at each point $p\in M$, where $Q_i=j^2f^i(0,0)$, $i=1,2$. (The coefficients of $Q_1$ and $Q_2$ depend on $p$.)

Representing a binary form $Ax^2+2Bxy+Cy^2$ by its coefficients 
$(A,B,C)\in \mathbb R^3$, there is a cone $B^2-AC=0$ representing the perfect 
squares. If the forms $Q_1$ and $Q_2$ are independent, they 
determine a line in the projective plane $\mathbb R P^2$ and the cone a conic. 
This line meets the conic in 0,1,2 points according as  $\delta(p)<0,=0,>0$, where
\[
\delta(p)=(an-cl)^2-4(am-bl)(bn-cm).
\]
A point $p$ is said to be {\it elliptic/parabolic/hyperbolic} if 
$\delta(p)<0/=0/>0$. The set of points $(x,y)$ where $\delta = 0$ is called the {\it parabolic set} of $M$. If $Q_1$ and $Q_2$ are dependent, the rank of the matrix
$$
A=\left(
\begin{array}{ccc}
a&b&c\\
l&m&n
\end{array}
\right)
$$
is 1 (provided either of the forms is non-zero); the corresponding points on the surface are referred to as {\it inflection points}. 

The solutions of the binary differential equation (BDE)
$$
(am-bl)dx^2 +(an-cl)dxdy+(bn-cm)dy^2 =0.
$$
is called {\it asymptotic directions}. This equation is affine invariant and the {\it discriminant curve} of the BDE coincides with the parabolic set. There are 2 asymptotic directions at a hyperbolic point, one at a parabolic point and none at an elliptic point.  At an inflection point all tangent directions are asymptotic directions.%Asymptotic curves determine a pair of regular foliations on the hyperbolic region of the surface. At generic points on the parabolic curve, they form a family of cusps with the cusps tracing the parabolic set. At isolated parabolic points, the configuration of the asymptotic curves is as in \cite{BruceTari}. At inflection points, the parabolic set generically has a Morse singularity.

%Following [Little, Basto], let $\mathcal M_1=\left(\begin{array}{cc}
%a&b\\
%b&c\\
%\end{array}
%\right)
%$ and $\mathcal M_2=\left(\begin{array}{cc}
%l&n\\
%n&m\\
%\end{array}
%\right)$ be the matrices associated to the above quadratic forms.
%$$
%\mathcal M_1=\left(\begin{array}{cc}
%a&b\\
%b&c\\
%\end{array}
%\right)
%\mbox{ and $\mathcal M_2=\left(\begin{array}{cc}
%l&n\\
%n&m\\
%\end{array}
%\right)$. }
%$$
The {\it Gaussian curvature} at $p$ is given
by
$$
 K(p) = ac-b^2 + ln-m^2
$$
%where  $K_i = 
%\det \mathcal M_i$ for $i = 1,2$,
and the {\it normal curvature} can be determined by
$$
\kappa_n(p)=(a-c)m-(l-n)b.
$$

In particular, the points where $\delta(p) = 0$ can be subclassified depending on the values of Gaussian curvature $K$ and normal curvature $\kappa_n$. In fact,  %when $\delta(p) = 0$ 
we can distinguish among the following possibilities:
\begin{itemize}
\item If $K(p) < 0$ and $\rank A= 2$, we said that $p$ is parabolic point. Here  $\kappa(p) \neq 0$;
\item If $K(p) < 0$ and $\rank A = 1$ then  $p$ is called by an {\it inflection point of real type}, and consequently, $\kappa(p) = 0$;
\item If $K(p) = 0$ then $p$ is an {\it inflection point of flat type} such that $\kappa(p) = 0$;
\item If $K(p) > 0$ then $p$ is an {\it inflection point of imaginary type}. In such case, $\kappa(p) = 0$. 
\end{itemize}

%There is an action of $GL(2,\mathbb R)\times GL(2,\mathbb R)$ on pairs of 
%binary forms. The orbits of this action are as follows (see for example \cite{gibson}):
%$$
%\begin{array}{ll}
%(x^2,y^2)&\mbox{hyperbolic point}\\
%(xy,x^2-y^2)&\mbox{elliptic point}\\
%(x^2,xy)&\mbox{parabolic point}\\
%(x^2\pm y^2,0)&\mbox{inflection point}\\
%(x^2,0)&\mbox{degenerate inflection point of the first type}\\
%(0,0)&\mbox{degenerate inflection point of the second type.}
%\end{array}
%$$

%\subsection{Non-torsal ruling} \label{non-torsal}

\subsection{Monge form of a smooth ruled surfaces in $4$-space}\label{prel}
Without loss of generality, we can take the parametrisation of the ruled surfaces (\ref{ruled}) with $F(0,0)=\bx(0)=\vec{0}$, the ruling
$R_0=\R \be(0)$ is the $y$-axis, and 
$\bx(u)$ and $\be(u)$ lie on 
the parallel hyperplanes $y=0$ and $y=1$, respectively. %(see Lemma 2.5 in \cite{MartinsNuno}). 
Denote by $x_i$ and $e_i$ the $i$-th component of $\bx$ and  $\be$, respectively.
Note that we can assume that $\bx(u)$ is non-singular at $u=0$,
because we deal with regular rulings. 
Thus we can set $x_1(u)=u$ taking a suitable coordinate.
Denote the Taylor expansions of $\bx(u)$ and $\be(u)$ at $u=0$ by 
\begin{equation*}\label{bxbe}
\bx(u)
=\left(
\begin{array}{c}
u\\
0\\
\ell_1u+\ell_2 u^2+\cdots\\
d_1u+d_2 u^2+\cdots
\end{array}
\right), 
\quad 
\be(u)=\left(
\begin{array}{c}
m_1u+m_2 u^2+\cdots\\
1\\
n_1u+n_2 u^2+\cdots \\
q_1u+q_2 u^2+\cdots 
\end{array}
\right). 
\end{equation*} 

Now, assume that $p=(0, t_0)$ 
then
by a linear change of the $xzw$-hyperplane, we can consider   $m_1=\ell_1=d_1=0$ and 
%\begin{equation*}
$\bx'(0)=(1,0,0,0)$,  $\be(0)=(0,1,0,0)$. %, \;\; \be'(0)=(0,0,n_1,1), 
%\end{equation*}
 Then 
\begin{equation}\label{eq1}
\left(
\begin{array}{c}
x\\
y\\
z\\
w
\end{array}
\right)
=
F(u,t)=\left(\begin{array}{c}
u+t(m_2 u^2+\cdots)\\
t\\
\ell_2 u^2+\cdots + t(n_1u+n_2 u^2+\cdots)\\
d_2u^2+\cdots+t(q_1u+q_2 u^2+\cdots)
\end{array}\right). 
\end{equation}
Now, we compute $F$ in Monge form at $p$. 
First, invert $x=x(u,t)$  in (\ref{eq1}), then
$u=x -m_2 x^2 t - m_3 x^3 t+\cdots + 2m_2^2 x^3t^2+\cdots$. 
Take a new $y$ so that $t=y+t_0$,
then 
$$
z=n_1t_0x+c_1x^2+n_1xy+o(2) \mbox{ and 
$w=q_1t_0x+c_2x^2+q_1xy+o(2)$}
$$
with $c_1=\ell_2 +n_2t_0 - m_2 n_1t_0^2$ and $c_2=d_2 +q_2 t_0 - m_2 q_1t_0^2$. 
After of a linear change at the target,   %$(x,y,z)\mapsto (x, y-cx,z+t_0x)$,
the Monge form $z=f^1(x,y)$ and $w=f^2(x,y)$ of $R$ at $p$ 
is expressed in the new coordinates $(x,y)$ centered at $p$ by  $\tilde{F}(x,y)=(x,y,f^1(x,y),f^2(x,y))$ with 
\begin{equation}\label{MongeGeral}
\displaystyle
\begin{array}{lll}
\displaystyle f^1(x,y)&=& a_{20}x^2+a_{11}xy+a_{30}x^3+a_{21}x^2y\\
&&\displaystyle +a_{40}x^4+a_{31}x^3y+a_{22}x^2y^2+\sum_{i+j\ge 5} a_{ij}x^i y^j \\
\displaystyle f^2(x,y)&=& b_{20}x^2+b_{11}xy+b_{30}x^3+b_{21}x^2y\\
&&\displaystyle +b_{40}x^4+b_{31}x^3y+b_{22}x^2y^2+\sum_{i+j\ge 5} b_{ij}x^i y^j 
\end{array}
\end{equation}
where $a_{ij},b_{ij}$ are polynomials of $\ell_i, m_j, n_k,d_i,q_i,t_0$, and $a_{ij}=b_{ij}=0$ when $i<j$. Here the second order coefficients are given by
$$
\begin{array}{ll}
a_{20}= -m_2n_1t_0^2+n_2t_0+\ell_2 & b_{20}=  -m_2q_1t_0^2+q_2t_0+d_2\\
a_{11}=n_1   & b_{11}=q_1 \\
a_{02}=0 & b_{02}=0.
\end{array}
$$
The higher order terms are given, for example, by
$$
\begin{array}{ll}
a_{30}=2m_2^2n_1t_0^3-(2m_2n_2+m_3n_1)t_0^2+(-2\ell_2m_2+n_3)t_0+\ell_3, &\\
b_{30}= 2m_2^2q_1t_0^3-(2m_2q_2+m_3q_1)t_0^2+(-2d_2m_2+q_3)t_0+d_3, &\\
 \end{array}
$$
$$
 \begin{array}{ll}
a_{21}=-2m_2n_1t_0+n_2, & b_{21}=-2m_2q_1t_0+q_2,\\
a_{22}=-m_2n_1, & b_{22}=-m_2q_1,\\
 \end{array}
 $$
 and so on. Note that $a_{11}b_{22}-b_{11}a_{22}=0$. 
 
 Furthermore, the coefficients of the second fundamental forms around $p$ are given by $a=\frac12 f^2_{xx}$, $b= \frac12f^2_{xy}$, $c=\frac12 f^2_{yy}$, $l=\frac12 f^1_{xx}$, $m=\frac12 f^1_{xy}$ and $n=\frac12 f^1_{yy}$.

%Let $p$ be a point on a ruled surface, and choose two smooth independent vector fields in the normal plane and two smooth independent tangent vector fields in a neighbourhood $U$ of $p$. This determines at each point near $p$ a system of coordinates where the surface is given locally in Monge form  (\ref{MongeGeral}).

 Then we have the following observation about ruled surfaces in 4-space.

%\begin{prop}\label{class}
%Any point on a ruled surface $R$ is either a parabolic point, an inflection point of real type or an inflection point of flat type. In particular, 
%\begin{itemize}
%\item[i)] if $\det([\bx'(0)\; \be(0)\; \be'(0)\;\be''(0)])\neq0$, on the ruling $R_0$ all points are parabolic except one which is an inflection point of real type.
%\item[ii)] if $\det([\bx'(0)\; \be(0)\; \be'(0)\;\be''(0)])=0$ at an inflection point and $n_2d_2-q_2\ell_2\neq0$, then $K=0$. 
%\item[iii)] if $K=0$, then all points on the ruling $R_0$ are inflections point of flat type.
%\end{itemize}
%\end{prop}

\begin{prop}\label{class}
Any point on a smooth ruled surface $R$ is either a parabolic point or an inflection point of real type. In particular, 
%\begin{itemize}
if $\det([\bx'(0)\; \be(0)\; \be'(0)\;\be''(0)])\neq0$, on the ruling $R_0$ all points are parabolic except one which is an inflection point of real type.
%\item[ii)] if $\det([\bx'(0)\; \be(0)\; \be'(0)\;\be''(0)])=0$ at an inflection point and $n_2d_2-q_2\ell_2\neq0$, then $K=0$. 
%\item[iii)] if $K=0$, then all points on the ruling $R_0$ are inflections point of flat type.
%\end{itemize}
\end{prop}

\begin{proof}
	Consider $p$ as above, then the matrix $A$ at $p$ is given by
$$
A=\left(
\begin{array}{ccc}
b_{20}&b_{11}&0\\
a_{20}&a_{11}&0
\end{array}
\right).
$$
We obtain ${\delta}(p)\equiv0$. On a regular ruling the Gaussian curvature is $K(p)=-(a_{11}^2+b_{11}^2)=-(n_{1}^2+q_{1}^2)<0$ and normal curvature is $\kappa_n(p)=b_{20}a_{11}-b_{11}a_{20}$. It is follows that the point $p$ is parabolic if and only if $\rank A=2$. That is equivalent to $\kappa_n(p)\neq0$. Note that  
\begin{eqnarray*}
b_{20}a_{11}-b_{11}a_{20}&=&(n_1q_2-n_2q_1)t_0+d_2n_1-\ell_2q_1\\
&=&\det([\bx'(0)\; \be(0)\; \be'(0)\;\be''(0)])t_0+\det([\bx'(0)\; \be(0)\; \be'(0)\;\bx''(0)])\neq0\\
&=&(a_{11}b_{21}-a_{21}b_{11})t_0+\det([\bx'(0)\; \be(0)\; \be'(0)\;\bx''(0)])\neq0.
\end{eqnarray*} 

In opposite case, i.e. $\kappa_n(p)=0$, $p$  is an inflection point of real type. %Finally $p$ is an inflection point of flat type when $K=0$ if and only if $a_{11}=b_{11}=0$. 
In particular, if $\det([\bx'(0)\; \be(0)\; \be'(0)\;\be''(0)])=n_1q_2-n_2q_1\neq 0$, the statement on $R_0$ follows from above expression. %$\det([\bx'(0)\; \be(0)\; \be'(0)\;\be''(0)])=n_1q_2-n_2q_1$.  %In particular, when $a_{11}b_{21}-a_{21}b_{11}=0$ we obtain that $a_{11}=b_{11}=0$, if $a_{20}b_{21}-a_{21}b_{20}=\ell_2q_2-d_2n_2\neq0$.  In this case, the origin is a degenerate inflection point of the first type.

\qed
\end{proof}

\begin{rem} At an inflection point of real type with $\det([\bx'(0)\; \be(0)\; \be'(0)\;\be''(0)])=0$  and $n_2d_2-q_2\ell_2\neq0$ implies the occurrence of $K=0$ such that all points of the ruling $R_0$ are inflection points of flat type, i.e.  the ruling $R_0$ is singular. However, this case does not happen because we only consider regular rulings. The ruled surface in 4-space space with $K = 0$ is acknowledged as a developable surface (\cite{Plass}). So, in this current study, our focus lies on the non-developable case, specifically on parabolic points and inflection points of real type. %It is worth noting that our concern throughout this paper will be smooth points on a ruled surface.
\end{rem}

We denote by $V_k$ the set of polynomials in $x$, $y$ of degree greater than or equal to $2$ and less than or equal to $k$. Following \cite{BruceTari} we obtain a smooth map, the Monge-Taylor map $G : R,p\to V_k\times V_k$, which associates to each point $q$ near $p$ the $k$-jet of the pair of functions $(f_1, f_2)$ defined above at the point $q$. The set $V_k\times V_k$ has a natural $\mathcal G = GL(2,R)\times GL(2,R)$-action given by linear change of coordinates in the tangent and normal plane. The flat geometry of smooth manifolds is affine invariant (see \cite{bgt1}). A subset
$Z$ of $V_k\times V_k$ that is of any geometric significance will be $\mathcal G$-invariant. Moreover, if $Z$ is furnished with a Whitney regular stratification, then for any generic ruled surface $R$ the map germ $R,p\to V_k\times V_k$ will be transverse to the strata of $Z$.
%Since a geometrical property of a ruled surface on the image of $F$ determines a smooth submanifold $Z$ in $V_k\times V_k$,
%then the map $\tilde{\mu}$ is transverse to these submanifolds $Z$.
%Using Thom's lemma %(\cite{BruceGiblin}, Theorem 8.17),
%it follows that for a generic ruling $(\bx, \be)$ the map $\mu_{(\bx, \be)}$ is transverse to $Q$.
According to this fact, we try to stratify the jet space of $G$ up to codimension $2$.
Note that the properties which induce submanifolds of codimension $>2$ never happen for generic ruled surface.

\begin{cor}\label{InfleCurve} 
For any generic ruled surface in 4-space there is a regular curve formed by inflection points of real type transverse to ruling $R_0$ that we call the {\normalfont inflection curve}. %In particular, the inflection curve is tangent to $x$-axis at a special isolated point $p$.
\end{cor}

\begin{proof}
Consider $t_0=0$, then we take $p=(0,0)$ an infection point of real type, then $d_2n_1-\ell_2q_1=0$ with $\det([\bx'(0)\; \be(0)\; \be'(0)\;\be''(0)])=a_{21}b_{11}-a_{11}b_{21}=n_1q_2-n_2q_1\neq0$. In this case, around $p$ we parametrise the set of the inflection points of real type by
\begin{eqnarray*}
\kappa_n(x,y)&=&\Big(2(a_{20}b_{21}-a_{21}b_{20})+3(a_{30}b_{11}-a_{11}b_{30})\Big)x+\Big(a_{21}b_{11}-a_{11}b_{21}\Big)y+\cdots\\
&=&\Big(2(\ell_2q_2-d_2n_2)+3(\ell_3q_1-d_3n_1)\Big)x+\Big(n_2q_1-n_1q_2\Big)y+\cdots=0
\end{eqnarray*}
Since the ruling $R_0$ is the $y$-axis we obtain the desired result. %If $2(a_{20}b_{21}-a_{21}b_{20})+3(a_{30}b_{11}-a_{11}b_{30})=0$ then the inflection curve is tangent to $x$-axis at isolated point $p$ when $$ 2(a_{11}b_{40}-a_{40}b_{11})+(a_{31}b_{20}-a_{20}b_{31})+2(a_{21}b_{30}-a_{30}b_{21})\neq0.$$
\qed
\end{proof}

%\medskip

%%%%%%%%%%%%%%%%%%%%%%%%
\section{Parallel projections along planes}\label{sec:singprojc}

We are interested in the affine geometry of ruled surfaces using contact with planes. The geometric characterization of points on a generic surface $M$ in 4-space obtained using pairs of quadratic forms can be recovered by considering the contact of the surface with lines and hyperplanes (see \cite{IFRT-book}). However, when making the contact between a ruled surface and these objects, away from inflection points, the contact is highly degenerate and does not exhibit any characteristics on the ruled surface. In fact, the contact with lines is given by $\mathcal A$-singularities of the orthogonal projection. This projection is singular when the line is in an asymptotic direction (\cite{bruce-nogueira}). As the ruling is an asymptotic direction, the projection along to it is $\mathcal A^{(k)}$-equivalent to $(x,0,xy)$, that is, a highly degenerate singularity of infinity codimension. As the contact with lines and hyperplanes is dual (\cite{bruce-nogueira}), the contact of the ruled surface with hyperplanes is also degenerate. Then we consider here the contact of ruled surfaces with 2-planes.
  
A plane $\pi_1$ is the kernel of a linear submersion $\psi:\mathbb R^4\to \mathbb \pi_2$, where $\pi_2$ is a transverse plane to $\pi_1$. Following Montaldi in \cite{Montaldi}, the contact of a surface $M$ with $\pi_1$ at a given point on $M$ is captured by the $\mathcal K$-singularities of the composite map 
$\psi\circ \varphi$, with $\varphi$ a local parametrisation of $M$. The map $\psi\circ \varphi$  is locally a map-germ $(\mathbb R^2,0)\to (\mathbb R^2,0)$. We consider here the action of the group $\mathcal A$ as it gives finer results and see the Rieger's  $\mathcal A$-classification for corank 1 case in Table \ref{tab:rieger}.

The set of all planes through the origin in $\mathbb R^4$ is the Grassmannian $Gr(2,4)$, which is a 4-dimensional manifold. We use below the family of parallel projections to transverse planes using only the affine structure of $\mathbb R^4$. We use the same construction as in \cite{DT}.

Consider two transverse planes, $\pi_1$ and $\pi_2$, in $\mathbb{R}^4$. We choose a basis vector ${\mathcal F}=\{ {\bf f_1},{\bf f_2},{\bf f_3},{\bf f_4}\}$ of $\mathbb R^4$ such that ${\bf f_1}$ and ${\bf f_2}$ generate $\pi_1$, while ${\bf f_2}$ and ${\bf f_3}$ generate $\pi_2$. Let $\pi$ be a plane 
near to $\pi_1$ that can be generated by the two column vectors ${\bf u}$ and ${\bf v} $ of a matrix 
%$$
%\left(
%\begin{array}{cc}
%\lambda_1&\mu_1\\
%\lambda_2&\mu_2\\
%\lambda_3&\mu_3\\
%\lambda_4&\mu_4
%\end{array}
%\right),
%$$
%with $\lambda_1\mu_2-\lambda_2\mu_1\ne 0$ for $\pi$ near $\pi_1$, and 
%where the coordinates of the vectors are with respect to the basis $\mathcal F$. 
%As the choice of generators is arbitrary, 
%the column vectors ${\bf u}$ and ${\bf v} $ of 
$$
%\left(
%\begin{array}{cc}
%\lambda_1&\mu_1\\
%\lambda_2&\mu_2\\
%\lambda_3&\mu_3\\
%\lambda_4&\mu_4
%\end{array}
%\right)\cdot{}\left(
%\begin{array}{cc}
%\lambda_1&\mu_1\\
%\lambda_2&\mu_2
%\end{array}
%\right)^{-1}
%=
\left(
\begin{array}{cc}
1&0\\
0&1\\
\alpha_1&\beta_1\\
\alpha_2&\beta_2
\end{array}
\right)
$$
%also generate the plane $\pi$. 
where the coordinates of the vectors are with respect to the basis $\mathcal F$. 
Therefore, we can identify the set of planes near $\pi_1$ with an open 
neighbourhood $U$ of the origin in  $\mathbb R^4$. 
The plane $\pi$ above, being close to $\pi_1$, is still transverse to $\pi_2$, so we project along $\pi$ to $\pi_2$ (see \cite{DT}) for details). 

The parallel projection of  ${\bf x}=(x_1,x_2,x_3,x_4)_{\mathcal F}\in \mathbb R^4$ to $\pi_2$ is a point 
${\bf y}={\bf x}-x_1 {\bf u}-x_2 {\bf v}\in \pi_2$. Therefore, the expression for the family of parallel projections 
$
P: \mathbb R^4\times U \to \pi_2\equiv \mathbb R^2 
$  
along planes near $\pi_1$ to the plane $\pi_2$, 
is given by
$$
P({\bf x}, \pi)=(x_3-\alpha_1x_1-\beta_1x_2,x_4-\alpha_2x_1-\beta_2x_2),
$$
with $\pi$ generated by ${\bf u},{\bf v}$ above. If we choose another basis for $\pi_2$ or another transverse plane to $\pi_1$, then the expressions for the projections are equivalent under 
multiplication by an invertible $2\times 2$ matrix, so the resulting projections are $\mathcal A$-equivalent. 

Restricting the family $P$ to a surface $M\subset \mathbb R^4$ gives  the family of parallel projections of $M$ to transverse planes. We still denote this restriction by $P$ and by $P_{\pi}:M\to \mathbb R^2$, the map 
given by $P_{\pi}(p)=P(p,\pi)$, for $p\in M$. We also denote by  
$j_1^kP(p,\pi)$, the $k$-jet of $P_{\pi}$ at $p$.  By Montaldi's theorem in \cite{Montaldi2} if
$W$ is an $\mathcal A$-invariant variety in $J^k(M,\mathbb R^2)$, then, for a residual set of immersions of a surface $M$ in $\mathbb R^4$, 
the map $j_1^kP: M\times Gr(2,4)\to 
J^k(M,\mathbb R^2)$ 
is transverse to $W$.
%\end{thm} 
This is means that  
%\begin{proof}
%The proof follows by Montaldi's theorem in \cite{Montaldi2} using 
%the fact that the family $P:\mathbb R^4\times Gr(2,4)\to \mathbb R^2$ is an $\mathcal A_e$-versal family.
%\end{proof}
%
%
%An important consequence of the Montaldi's type theorem above is that, 
for a generic surface $M$, 
the only singularities that $P_{\pi}$ can have are those 
of {$\mathcal A_e$-codimension $\le 4=\dim(Gr(2,4))$} (of the stratum, in the presence of moduli). Furthermore, these singularities are $\mathcal A_e$-versally unfolded by the family $P$.
%When $P_{\pi}$ is of corank 1, its possible local singularities are those in Table \ref{tab:rieger}.

%We reproduce in Table \ref{tab:rieger} a list of representatives of finitely $\mathcal A$-determined corank 1 map-germs of  $\mathcal A_e$-codimension $\le 4$ from \cite{rieger}.

\begin{table}
\caption{$\mathcal A$-singularities of corank 1 map-germs with $d_e(g,\mathcal A)\le 4$ (\cite{rieger}).}
\begin{center}
\begin{tabular}{l l c}
\hline
Type & Normal form & $d_e(g,\mathcal A)$ \cr
\hline
$1$ (Immersion) & $(x,y)$ &$0$\cr
$2$ (Fold) & $(x,y^2)$ &$0$\cr
$3$ (Cusp) & $(x, xy + y^{3})$ & $0$ \cr
$4_k$ (Lips/beaks for $k=2$) & $(x,y^3 \pm x^{k}y), 2\leq k\leq5$ & $k-1$\cr
$5$ (Swallowtail) & $(x, xy+ y^4)$ & $1$\cr
$6$ (Butterfly) & $(x,xy+y^{5}\pm y^7)$ & $2$\cr
$7$ & $(x, xy+ y^5)$ & $3$\cr
$8$ & $(x,xy+y^{6}\pm y^8+\alpha y^9)$ & $4(3^*)$\cr
$9$ & $(x, xy+ y^6+y^9)$ & $4$\cr
$10$ & $(x,xy+y^{7}\pm y^9+\alpha y^{10}+\beta y^{11})$ & $6(4^*)$\cr
$11_{2k+1}$ & $(x,xy^2+y^4+y^{2k+1}), 2\leq k\leq4$ & $k$\cr
$12$ & $(x,xy^2+y^{5}+ y^6)$ & $3$\cr
$13$ & $(x, xy^2+ y^5\pm y^9)$ & $4$\cr
$15$ & $(x,xy^2+y^{6}+y^7\pm \alpha y^9)$ & $5(4^*)$\cr
$16$ & $(x, x^2y+ y^4\pm y^5)$ & $3$\cr
$17$ & $(x, x^2y+ y^4)$ & $4$\cr
$18$ & $(x,x^2y+xy^3+\alpha y^{5}+y^6\pm \beta y^7)$ & $6(4^*)$\cr
$19$ & $(x,x^3y+\alpha x^2y^2+ y^{4}+x^3y^2)$ & $5(4^*)$\cr
\hline
\end{tabular}

$(*)$: the codimension in brackets is that is of the $\mathcal A$-stratum.
\end{center}
 \label{tab:rieger}
\end{table}

%\medskip
We consider  now the case when $M$ is a ruled surface $R$ and analyse 
in this section the $\mathcal A$-singularities of the germs, at a fixed point $p\in R$, of 
the parallel projections of $R$ along planes.  We denote by
${\bf E = \{e_1, e_2, e_3, e_4\}}$ the standard bases in $\mathbb R^4$ then the ruled surface $R$ can be parametrised in Monge form 
$$
\tilde{F}(x,y)=(x,y,f^1(x,y),f^2(x,y)),
$$
where $f^1,f^2$ are given as in \eqref{MongeGeral}, and with $T_pR$  generated by ${\bf e_1}$ and ${\bf e_2}$. 
 
It is not hard to show that $P_{\pi}$ is singular at the origin if and only if $\dim(T_pR\cap \pi)\ge 1$. 
We consider  the following two cases separately. 

\medskip
\noindent
\underline{\it Case $1$}: $\dim(T_pR\cap \pi)=1$. \\
As we have fixed the basis $\bf E$ and $T_pR$, 
we change the earlier setting for the parallel 
projection and take a pair of generators ${\bf u},{\bf v}$ of $\pi\in Gr(2,4)$ in general form.

We take, without loss of generality ${\bf u}=(1,\alpha,0,0)\in T_pR\cap \pi$ %(the case ${\bf u}=(0,1,0,0)$ does not give new information) 
and ${\bf v}=(0,\lambda,\mu,\beta)$, with 
$\mu^2+\beta^2\ne 0$ as $\dim(T_pM\cap \pi)=1$. 
We consider the parallel projection along $\pi$ to the transverse plane $\pi_2$ generated by 
${\bf e_2}$ and ${\bf f}=(1,0,-\beta,\mu)$.  
Then, 
$$
P_{\pi}({\bf x}){\sim_{\mathcal A}}\left(x_2-\alpha x_1-\frac{\lambda}{\mu^2+\beta^2} (\mu x_3+\beta x_4),\, \beta x_3-\mu x_4\right),
$$
where we scaled the last coordinate by $\mu^2+\beta^2$, and where the coordinates of $P_{\pi}({\bf x})$ are with respect to the basis $\{{\bf e_2} ,{\bf f}\}$ of $\pi_2$. 
 The restriction of $P_{\pi}$ to the surface $R$ (still denoted by $P_{\pi}$) 
is the map-germ $P_{\pi}:(\mathbb R^2,0)\to (\mathbb R^2,0)$ given by

\begin{equation}\label{eq:case1}
P_{\pi}(x,y){\sim_{\mathcal A}}\left(y-\alpha x-\frac{\lambda}{\mu^2+\beta^2} (\mu f^1(x,y)+\beta f^2(x,y)),\, \beta f^1(x,y)-\mu f^2(x,y)\right).
\end{equation}

It is clear that $P_{\pi}$ is a corank 1 map-germ. Because $R$ is a ruled surface, it is not necessarily in the residual set, so we should not expect all the singularities in Table 1 to occur for $P_{\pi}$. Indeed,

\begin{prop} \label{prop:SingSigma}
With notation as above, let $R$ be a ruled surface and $\pi\in Gr(2,4)$ with $\dim(T_pR\cap \pi)=1$. 
\begin{enumerate}
\item[(i)] At a parabolic point,  the singular set of the projection $P_{\pi}$ is a regular curve such that $j^2P_{\pi}$ is $\mathcal A^{(2)}$-equivalent to $(y,x^2)$ or to  $(y,xy)$. 
\item[(ii)] At an inflection point of real type, almost all planes $\pi$ the $j^2P_{\pi}$ is $\mathcal A^{(2)}$-equivalent to $(y,x^2)$ or to  $(y,xy)$ except some special planes $\pi$ such that the singular set of $P_{\pi}$ has a Morse $A_1^-$-singularity. In this case,  $j^3P_{\pi}$ is $\mathcal A^{(3)}$-equivalent to $(y,x^3-xy^2)$ or to  $(y,xy^2)$.  
\end{enumerate} 
\end{prop}

\begin{proof} 
 The singular set of $P_{\pi}$ as in (\ref{eq:case1}) is the zero set of the function
$$
g(x,y)=\mu (f^1_x+\alpha f^1_y)-\beta (f^2_x+\alpha f^2_y)-\lambda(f^1_xf^2_y-f^1_yf^2_x).
$$
The result follows by considering the relevant jet of $g$ at the origin. 

$(i)$ We have 
\begin{eqnarray*}
\frac12j^1g(x,y)&=&\Big(\frac{\alpha}{2}(b_{11}\mu-a_{11}\beta)+(b_{20}\mu-a_{20}\beta)\Big)x+ \frac12\Big(b_{11}\mu-a_{11}\beta\Big)y\\
&=&\frac12 \Big(2m_2(\beta n_1-\mu q_1)t_0^2+2(\mu q_2-\beta n_2)t_0+\alpha( \mu q_1-\beta n_1)-2\beta\ell_2+2\mu d_2\Big)x\\
&&+ \frac12\Big(-n_1\beta+q_1\mu\Big)y.
\end{eqnarray*}
So when the origin is a parabolic point ($a_{20}b_{11}-a_{11}b_{20}\ne 0$), $j^1g$ is identically zero if and only if $\mu=\beta=0$. But this 
is equivalent to $\pi=T_pM$ ($\dim(T_pM\cap \pi)=2$). Then, 
the singular set is a regular curve. Consequently, $j^2P_{\pi}$ is $\mathcal A^{(2)}$-equivalent to $(y,x^2)$ or to  $(y,xy)$.

$(ii)$ Consider the 2-jet of $g$ at the origin 
\begin{eqnarray*}
j^2g(x,y)&=&\Big((\alpha b_{11}+2b_{20})\mu-(a_{11}\alpha+2a_{20})\beta\Big)x+ \Big(b_{11}\mu-a_{11}\beta\Big)y\\
&&+\Big(2(a_{20}b_{11}-a_{11}b_{20})\lambda   +\alpha( b_{21}\mu-a_{21}\beta)    +3(b_{30}\mu-a_{30}\beta)\Big)x^2\\
&&+2(b_{21}\mu-a_{21}\beta)xy.
\end{eqnarray*}
At an inflection point of real type (i.e. $a_{20}b_{11}-a_{11}b_{20}=0$), we need to consider the cases: $b_{11}\neq0$ (the case $a_{11}\neq0$ is analogous) and $a_{20}\neq0$ (the case $b_{20}\neq0$ is analogous). %The other cases when $a_{11}\neq0$ and $b_{20}\neq0$ are analogous. 

If $b_{11}\neq0$ then, $a_{20}=\displaystyle \frac{a_{11}b_{20}}{b_{11}}$. Thus 
\begin{eqnarray}
j^1g(x,y)&=&-\frac{(\alpha b_{11}+2b_{20})(a_{11}\beta-b_{11}\mu)}{b_{11}}x -(a_{11}\beta-b_{11}\mu)y.
\end{eqnarray}
Here the $j^1g$ is non zero
if and only if   $(\beta,\mu)\neq(b_{20}, a_{20})$. 

Similarly if $a_{20}\neq0$ we obtain  $b_{11}=\displaystyle \frac{a_{11}b_{20}}{a_{20}}$, then   
\begin{eqnarray}
j^1g(x,y)&=&-\frac{(a_{20}\beta-b_{20}\mu)(a_{11}\alpha+2a_{20})}{a_{20}}x-\frac{a_{11}(a_{20}\beta-b_{20}\mu)}{a_{20}}y.
\end{eqnarray}
Again the $j^1g$ is non zero if and only if   $(\beta,\mu)\neq(b_{11}, a_{11})$. Then $j^2P_{\pi}$ is $\mathcal A^{(2)}$-equivalent to $(y,x^2)$ or to  $(y,xy)$.

When $(\beta,\mu)=(b_{20}, a_{20})$, 
%If $a_{20}\neq0$ then $j^1g$ is identically zero if and only if $(\beta,\mu)=(b_{11}, a_{11})$.
%we have that $j^1g$ is identically zero if and only if $(\beta,\mu)=(b_{20}, a_{20})$ or $(\beta,\mu)=(b_{11}, a_{11})$. 
the discriminant of the quadratic form $ j^2g(x,y)$
is equal to 
\begin{eqnarray*}
D_1&=&\frac{4b_{20}^2(a_{11}b_{21}-a_{21}b_{11})^2}{b_{11}^2}=\frac{4(-m_2q_1t_0^2+q_2t_0+d_2)^2(n_1q_2-n_2q_1)^2}{q_1^2} 
%&=& 4\left(\frac{m_2(d_2n_1-\ell_2q_1)^2+(n_1q_2-n_2q_1)(d_2n_2-\ell_2q_2)}{n_1q_2-n_2q_1}\right)^2
\end{eqnarray*}

If $(\beta,\mu)=(b_{11}, a_{11})$, 
the discriminant of the quadratic form $ j^2g(x,y)$
is equal to 
$$
D_2=\displaystyle\frac{4a_{11}^2(a_{20}b_{21}-a_{21}b_{20})^2}{a_{20}^2}=\frac{4n_1^2(m_2(n_1q_2-n_2q_1)t_0^2+2m_2(d_2n_1-\ell_2q_1)t_0-d_2n_2+\ell_2q_2)^2}{(-m_2n_1t_0^2+n_2t_0+\ell_2)^2}
$$
As $D_1$ and $D_2$ are positive the statement follows. In particular, by Proposition \ref{class} and considering $(n_1q_2-n_2q_1)\neq0$, we obtain %on the ruling $R_0$ 
at $p=\left(0,-\frac{d_2n_1-\ell_2q_1}{n_1q_2-n_2q_1}\right)$ that %$t_0=-\displaystyle\frac{d_2n_1-\ell_2q_1}{n_1q_2-n_2q_1}$, then 
$$
D_1=\displaystyle4\left(\frac{m_2(d_2n_1-\ell_2q_1)^2+(n_1q_2-n_2q_1)(d_2n_2-\ell_2q_2)}{n_1q_2-n_2q_1}\right)^2 \mbox{ and $D_2=4(n_1q_2-n_2q_1)^2$.}
$$
%When $a_{11}b_{21}-a_{21}b_{11}=\det([\bx'(0)\; \be(0)\; \be'(0)\;\be''(0)])\neq0$ (resp. $(a_{20}b_{21}-a_{21}b_{20})\neq0$), the discriminant is positive, 
In both cases $g$ has an $A_1^-$-singularity, that is, it  is $\mathcal A$-equivalent to $x^2-y^2$. Consequently, $j^3P_{\pi}$ is $\mathcal A^{(3)}$-equivalent to $(y,x^3-xy^2)$ or to  $(y,xy^2)$.   

%At a inflection point of flat type, ($a_{11}=b_{11}=0$) the analysis is similar to inflection point real type case and it does not give extra information.% $j^1g$ is non zero if and only if  $(\beta,\mu)\neq(b_{20}, a_{20})$. Again  with $a_{20}b_{21}-a_{21}b_{20}\neq0$) the $j^4P_{\pi}$ is $\mathcal A^{(4)}$-equivalent to $(y,x^3)$. 
\qed
\end{proof}

%\begin{rem}
%Note that $D_1=0$ if and only if $b_{20}=0$ or $(a_{11}b_{21}-a_{21}b_{11})=0$. But 
%\end{rem}

\

\noindent
\underline{\it Case $2$}: $\pi=T_pM$ \\
We project to the plane $\pi_2$ generated by ${\bf e_3}$ and ${\bf e_4}$. Then 
$
P_{\pi}({\bf x})=(x_3,x_4)
$
and its restriction to the surface $R$ is the corank 2 map-germ 
\begin{equation}\label{eq:case2}
P_{\pi}(x,y)=(f^1(x,y),f^2(x,y)).
\end{equation}

%Here, we are essentially considering the contact of the curve $\gamma$ with its tangent line. 

\medskip
According to Proposition \ref{class}, the origin is a parabolic point if and only if 
$a_{20}b_{11}-a_{11}b_{20}\ne 0$. Otherwise, it is an inflection point of real type.
We shall analyse the singularities of the germs (\ref{eq:case1}) and (\ref{eq:case2}) in two case: when the origin 
is a parabolic point and in the second case when it not is.

% % % % % % % % % % % % % % % % % %
\subsection{At a parabolic point}

In this case, $a_{20}b_{11}-a_{11}b_{20}\ne 0$. As ${\bf e}'(0)\neq\vec{0}$ then we can consider $b_{11}=q_1=1$  (the case with $a_{11}=n_{1}=1$ is similar and does not give extra information). After of linear changes of the coordinate %and still considering the same notation before, 
we can write %$(x,y,z)\mapsto (x, y-cx,z+t_0x)$,
the Monge form $z=f^1(x,y)$ and $w=f^2(x,y)$ of $R$ at $p$ as
\begin{equation}\label{parabolicMonge}
\displaystyle
\begin{array}{l}
\displaystyle f^1(x,y)= \bar{a}_{20}x^2+\bar{a}_{30}x^3+\bar{a}_{21}x^2y+\bar{a}_{40}x^4+\bar{a}_{31}x^3y%a_{22}x^2y^2
+\sum_{i+j\ge 5} \bar{a}_{ij}x^i y^j \\
\displaystyle f^2(x,y)=xy+\bar{b}_{30}x^3+\bar{b}_{21}x^2y+\bar{b}_{40}x^4+\bar{b}_{31}x^3y+\bar{b}_{22}x^2y^2+\sum_{i+j\ge 5} \bar{b}_{ij}x^i y^j 
\end{array}
\end{equation}
with  
\begin{eqnarray*}
\bar{a}_{20}&=& (n_1q_2-n_2)t_0+d_2n_1-\ell_2, \qquad \bar{a}_{21}=n_1q_2-n_2, \\
%a_{30}&=&m_2(n_1q_2-n_2)t_0^2+(2d_2m_2n_1+n_1q_2^2-2\ell_2 m_2-n_1q_3-n_2q_2+n_3)t_0\\
%& &+d_2n_1q_2-d_2n_2-d_3n_1+\ell_3.\\
\bar{b}_{30}&=&  d_3 - d_2 q_2 + (- q_2^2 + q_3) t_0 + (m_2 q_2- m_3 ) t_0^2,\\
\bar{b}_{21}&=&q_2 - 2 m_2 t_0, \qquad  \bar{b}_{22}=-m_2, \\
\bar{b}_{31}&=&
q_3 - 2( m_3 + m_2 q_2) t_0 +  4 m_2^2 t_0^2, 
%b_{32}&=&- m_3 - 2 m_2 q_2 + 6 m_2^2 t_0. 
 \end{eqnarray*}
and $\bar{a}_{ij}=0$ when $i=j$. As the contact of $R$ with planes is affine invariant (\cite{bgt1}), we can set $\bar{a}_{20}=1$. 
Observe that we can also set, without loss of generality, $\beta =1$ in expression (\ref{eq:case1}) of $P_{\pi}$ 
(setting $\mu=1$ when $\beta=0$ does not give extra information). 
Then, the set $\Pi\subset Gr(2,4)$ of planes $\pi$ for which $P_{\pi}$ is singular at the origin 
can be parametrised by $(\alpha,\lambda,\mu)\in \mathbb R^3$. 

We have the following about the $\mathcal A$-singularities of $P_{\pi}$, where the type number refers to one in the first column of Table \ref{tab:rieger}. %The exact meaning of the term generic is given later (Theorem \ref{theo:genericityP}). The term worse means more degenerate (i.e.,  of higher $\mathcal A_e$-codimension). 
   
\begin{thm}\label{theo:notInf}
Let $M$ and $\Pi$ be as above.

\noindent 
{\rm (a)} Suppose that $\dim (T_pM\cap \pi)=1$.
 
{\rm (a1)} For $\pi \in \Pi$ and away from a surface $S \subset \Pi$, the singularity of the map-germ $P_{\pi}$ is of Type 2 {\rm (}fold{\rm)}. 

{\rm (a2)}  For $\pi\in S$ and away from a curve $C$ on $S$, the singularity of $P_{\pi}$ is of Type 3 {\rm (}cusp{\rm)}. 

{\rm (a3)}  For $\pi \in C$ %\setminus \{\pi'_1,\pi'_2\}$ 
except possibly at most two planes $\pi'_1,\pi'_2$ %and away from a curve $B$ on $M$ {\rm (}see {\rm (a5)}{\rm)}, 
the singularity of $P_{\pi}$ is of Type 5 {\rm (}swallowtail{\rm)}. 

{\rm (a4)}  The planes $\pi'_1$ and $\pi'_2$ in {\rm (a3)} are determined by two directions in $T_pM$, which we call {\rm butterfly directions}, given by
\begin{equation}\label{ButterflyDirections}
A\alpha^2+B\alpha-C=0
\end{equation}
where
\begin{align*}\label{coef-BDE}
\left\{
\begin{array}{lll}
A=\bar{a}_{31}-\bar{a}_{21}\bar{a}_{30}-\bar{a}_{21}\bar{b}_{21}-\bar{b}_{22} \\
B=\bar{a}_{40}-\bar{a}_{30}^2+\bar{b}_{21}^2-\bar{b}_{31}\\
C =\bar{b}_{40}-\bar{b}_{21}\bar{b}_{30}-\bar{a}_{30}\bar{b}_{30}\\
\end{array}
\right.
\end{align*}
provided the discriminant of equation (\ref{ButterflyDirections}) is non negative. 
The singularity of $P_{\pi'_i}, i=1,2,$ is of Type 6 {\rm (}butterfly{\rm)} or worse. 
There is a curve on the surface $R$ {\rm (}which may be empty{\rm )} where the singularity 
of $P_{\pi'_i},i=1,2,$ becomes of Type 7. There is also another curve {\rm (}possibly empty{\rm )}
on $R$ where the singularity becomes of Type 8  and possible isolated points on this curve where it degenerates further to Type 9 or Type 10. 

{\rm (a5)} Generically, there is a curve on the surface $R$ when the discriminant of equation {\rm {\rm (\ref{ButterflyDirections})}} vanishes which we call {\rm butterfly parabolic curve}. In particular, on the ruling $R_0$ there are at most $6$ points that are {\rm butterfly parabolic points} (i.e. points where the discriminant of (\ref{ButterflyDirections}) is zero).  
%umbilic points, the singularities of $P_{\pi}$ with  $\pi \in C$ are of Type 6 except for at most 12 planes where it
%is of Type 7, and exactly 3 planes on $C$ where it is of Type 8.

\smallskip
\noindent
{\rm (b)} Suppose that $\pi=T_pR$. Then $j^kP_{\pi}$ is $\mathcal A^{(k)}$-equivalent to $(x^2,xy)$.
%It is not finitely $\mathcal A$-determined.
\end{thm}

\begin{proof}
For (a) we take $P_{\pi}$ as in (\ref{eq:case1}) with $\beta=1$ and use the Monge form as in (\ref{parabolicMonge}).
We change coordinates in the source so that 
$$
P_{\pi}(x,y)\sim_{\mathcal A} (y,g_{(\alpha,\lambda, \mu)}(x,y)).
$$

We shall write $g=g_{(\alpha,\lambda, \mu)}$.
The singularities we are seeking in Table \ref{tab:rieger} depend on certain jets of $g$
which we calculate using Maple. The task then becomes a problem of recognition of singularities
of map-germs from the plane to the plane. We use the criteria in Table 6.1 in \cite{Kabata,IFRT-book}.

(a1) We find that 
$$
j^2g=(\mu\alpha-1)x^2-\mu xy.
$$

Therefore, the singularity of $P_{\pi}$ is a fold unless the coefficient of $x^2$ in  $j^2g$ is zero, 
that is, $\mu\alpha-1=0$. This is the equation of the surface $S$, which is the product of a hyperbola with a line.

(a2) On $S$, we can write $\mu=\displaystyle\frac{1}{\alpha}$. Then, the coefficient of $xy$ in 
$j^2g$ becomes $-\displaystyle\frac{1}{\alpha}$ and is never zero, so we get singularities in Table \ref{tab:rieger} with a 2-jet 
$\mathcal A^{(2)}$-equivalent to $(y,xy)$, that is, $g$ has a regular singular set (see Proposition \ref{prop:SingSigma}(i)).

For $\pi\in S$, the singularity is a cusp unless the coefficient of $x^3$ in the Taylor expansion of $g$ is zero, that is, $\bar{a}_{21}\alpha^2-\lambda\alpha+(-\bar{b}_{21}+\bar{a}_{30})\alpha-\bar{b}_{30}=0$. If this happens, 
we can solve for $\lambda$ because $\alpha\ne 0$ on $S$ and get a curve $C\subset S$, parametrised by $(\alpha, \lambda(\alpha),\mu(\alpha))$, where the
singularity of $P_{\pi}$ is more degenerate than cusp.

(a3) For $\pi\in C$, the singularity is a swallowtail unless the coefficient of $x^4$ in the Taylor expansion of 
$g$ is zero, that is, 
$$
(\bar{a}_{31}-\bar{b}_{22}-\bar{a}_{21}\bar{b}_{21}-\bar{a}_{21}\bar{a}_{30})\alpha^2+
(\bar{b}_{21}^2-\bar{a}_{30}^2-\bar{b}_{31}+\bar{a}_{40})\alpha-
(\bar{b}_{40}-\bar{b}_{21}\bar{b}_{30}-\bar{a}_{30}\bar{b}_{30})=0
$$
Generically the coefficients of the above equation are not all zero, then it has at most two (resp. one) distinct solutions $\alpha_i,i=1,2$ which determine two (resp. one) planes $\pi'_i, i=1,2$ on $C$ when the discriminant $\Delta$ of the equation in $\alpha$ is positive (resp. $=0$). 

(a4)  The singularity of $P_{\pi'_i}$, $i=1,2$, is of Type 6 or 7 
when the coefficients $D_1(\alpha_i)$ of $x^5$ in the Taylor expansion of $g$ is not zero, where $D_1(\alpha)$ is a polynomial of degree 5.
%$$
%\begin{array}{rcl}
%D_1(\alpha_i)&=&-2a_3(5a_3b_3-2b_4)\alpha_i^4+(6a_3^3-23a_3b_3^2-2a_3a_4+4b_3b_4-2a_5)\alpha_i^3\\
%&&+(14a_3^2b_3-13b_3^3+8a_3b_4-2a_4b_3-6b_5)\alpha_i^2\\&&+(4a_3b_3^2-3a3^3-6a_3a_4+8b_3b_4+6a_5)\alpha_i
%-3a_3^2b_3-4b_3^3-6a_4b_3+2b_5.
%\end{array}
%$$
In this case, according to the criteria in Table 6.1 in \cite{IFRT-book}, and denoting by $g_{ki}$ the coefficient of $x^{k-i}y^i$ 
in the Taylor expansion of $g$, the singularity is of Type 7 when
\begin{equation*}
\label{eq:Type7}
(8g_{50}g_{70}-5g_{60}^2)g_{21}+2g_{50}(g_{31}g_{60}-20g_{41}g_{50})g_{21}+35g_{31}^2g_{50}^2=0,
\end{equation*}
otherwise it is of Type 6. 
The Type 7 singularities occur along a regular curve on $M$, with equation $D_2(\alpha_i)x+D_3(\alpha_i)y+O(2)=0$, where $D_2(\alpha)$ and $D_3(\alpha)$ are polynomial in $\alpha$. 
 
Suppose now that $D_1(\alpha_i)=0$. 
For a generic ruled surface $R$, the resultant, with respect to $\alpha$, 
of $D_1(\alpha)$  and of the left hand side of Equation (\ref{ButterflyDirections}) 
vanishes along a regular curve on $R$.
On this curve, the singularity of $P_{\pi'_i}$ is of Type 8 except maybe at isolated points 
where it becomes of Type 9 or Type 10.

(a5) The butterfly direction are determined by solution of Equation (\ref{ButterflyDirections}), then denoting by
$A=\bar{a}_{31}-\bar{a}_{21}\bar{a}_{30}-\bar{a}_{21}\bar{b}_{21}-\bar{b}_{22}$, 
$B=\bar{a}_{40}-\bar{a}_{30}^2+\bar{b}_{21}^2-\bar{b}_{31}$, and
$C =\bar{b}_{40}-\bar{b}_{21}\bar{b}_{30}-\bar{a}_{30}\bar{b}_{30}$ 
its discriminant is determined by $\Delta=B^2+4AC$. Generically, $\Delta=0$, i.e., the discriminant curve is a regular curve on ruled surface $R$ and on the ruling $R_0$, using the coefficients of the Monge form (\ref{parabolicMonge}),  it is a polynomial of degree 6 in $t_0$.

\smallskip
(b) It is easy to see that $j^kP_{\pi}\sim_{\mathcal A^{(k)}} (x^2,xy)$. 
%It follows that $P_{\pi}$ is 2-$\mathcal K$-determined (\cite{dimcagibson,gibson}), but it is not  finitely $\mathcal A$-determined.

\qed
\end{proof}

\

%By Theorem \ref{theo:notInf} the butterfly direction are determined by solution of Equation (\ref{ButterflyDirections}), then denoting by
%\begin{align}\label{hyp-coef-reg}
%\left\{
%\begin{array}{lll}
%A=\bar{a}_{31}-\bar{a}_{21}\bar{a}_{30}-\bar{a}_{21}\bar{b}_{21}-\bar{b}_{22} \\
%B=\bar{a}_{40}-\bar{a}_{30}^2+\bar{b}_{21}^2-\bar{b}_{31}\\
%C =-\bar{b}_{40}+\bar{a}_{30}\bar{b}_{30}+\bar{b}_{21}\bar{b}_{30}\\
%\end{array}
%\right.
%\end{align}
%the coefficients of equation we can defined the following.

%On a generic surface in $\mathbb R^4$, it is known that there are two asymptotic directions (or one, or none) at hyperbolic points (respectively, parabolic points or elliptic points). On a smooth ruled surface in $\mathbb R^4$, at parabolic points, it is possible to identify 2, 1, or 0 butterfly directions, depending on the sign of the discriminant of Equation (\ref{ButterflyDirections}). In this context, considering only the quantity of directions that may exist at each point, we can establish the following definition: 
%
%\begin{definition} \label{ButterflyPoints}
%Away from an inflection point of the ruled surface $R$, we said that $p$ is a {\it butterfly hyperbolic/parabolic/elliptic} point when the discriminant $B^2-4AC>0/=0/<0$ at $p$.
%\end{definition}
%
%
%We study more carefully these point in Section \ref{secBDE}.

% % % % % % % % % % % %
\subsection{At an inflection point of real type}

In this case, $a_{20}b_{11}-a_{11}b_{20}=0$. %with $a_{21}b_{11}-a_{11}b_{21}\neq0$. 
%After of linear changes of the coordinate and still considering the same notation the Monge form as in (\ref{parabolicMonge}), we can %write %$(x,y,z)\mapsto (x, y-cx,z+t_0x)$,
 %at an inflection point of real type $p$ 
% putting $\bar{a}_{20}= (n_1q_2-n_2)t_0+d_2n_1-\ell_2=0$. %Then the inflection point of real type $p$ on the ruling $R_0$ is given by
%$$
%t_{0}=-\frac{(d_2n_1-\ell_2)}{(n_1q_2-n_2)}
%$$ 
%if $\bar{a}_{21}=n_1q_2-n_2\neq0$. %From Corollary \ref{InfleCurve} the set of all ruling on $R$ form a set of inflection points of real type that is a curve the we call the {\it inflection curve} on $R$. 
%We follow the setting for Theorem \ref{theo:notInf}. In particular, we set $\beta=1$ in the expression (\ref{eq:case1}) of $P_{\pi}$.  
By Proposition \ref{prop:SingSigma}$(ii)$ we should analyze the $\mathcal A$-singularities of the expression (\ref{eq:case1}) of $P_\pi$ when $\pi$ is special or not so on.  
%We start with the case where $\pi\in Gr(2,4)$ is not special. In one case, it is necessary to examine the condition $(\beta,\mu)\neq(b_{20}, a_{20})$ with $b_{11}\neq0$. In another case, we consider $(\beta,\mu)\neq(b_{11}, a_{11})$ with $a_{20}\neq0$. We present bellow on the first case; the second case is analogous.
For $\pi\in Gr(2,4)$ is not special, it is necessary to examine some cases. Here, we present only the case when $b_{11}\neq0$; the other cases are analogous.  %When $\pi\in Gr(2,4)$ is special we get to analyze the projection of simultaneous way in such case as follows:

%For $\pi\in Gr(2,4)$ is not special it is necessary to examine the condition $(\beta,\mu)\neq(b_{20}, a_{20})$ with $b_{11}\neq0$ and  $(\beta,\mu)\neq(b_{11}, a_{11})$ with $a_{20}\neq0$. Here, we present only the first case; the second case is analogous. When $\pi\in Gr(2,4)$ is special we get to analyze of simultaneous way as follows:

After applying linear coordinate changes and still considering the same notation the Monge form as in (\ref{parabolicMonge}), the ruled surface $R$ can set $\bar{a}_{20}= (n_1q_2-n_2)t_0+d_2n_1-\ell_2=0$ at an inflection point of real type. %Consequently, we can assume   ${\bf u}=(1,\alpha,0,0)$ and ${\bf v}=(0,\lambda,\mu,1)$ with $\mu\neq 0$.
Then we obtain that
$$
j^2P_\pi(x,y)\sim_\mathcal A(y,\alpha\mu x^2+\mu xy)
$$
for $\pi\in Gr(2,4)$ (special or not). In particular, the set of planes $\Pi$ is parametrised by $(\alpha,\lambda,\mu)$ i.e. we can take $\beta=1$. Consequently, in this new coordinate change, the plane $\pi$ is not special when $\mu\neq0$; in otherwise, $\mu=0$.

\begin{thm} \label{theo:inf1}
Let $R$ and $\Pi$ be as above and suppose that $p$ 
an inflection point of real type at the origin with $\mu\neq0$. 

\noindent 
{\rm (a)} Suppose that $\dim (T_pR\cap \pi)=1$.

{\rm (a1)} For $\pi \in \Pi$ and away from a surface $S \subset \Pi$, the singularity of the map-germ $P_{\pi}$ is of Type 2 {\rm (}fold{\rm)}. 

{\rm (a2)}  For $\pi\in S$ and away from a curve $C$ on $S$, the singularity of $P_{\pi}$ is of Type 3 {\rm (}cusp{\rm)}. 

{\rm (a3)}  For $\pi \in C\setminus \{\pi'_1\}$ %and away from a curve $B$ on $M$ {\rm (}see {\rm (a5)}{\rm)}, 
the singularity of $P_{\pi}$ is of Type 5 {\rm (}swallowtail{\rm)}. 

{\rm (a4)}  At $\pi'_1$ the singularity of $P_{\pi'_1}$ is of Type 6 {\rm (}butterfly{\rm)} or worse. 
There is a isolated point at the surface $R$ where the singularity 
of $P_{\pi'_1}$ becomes of Type 7. There is also another isolated point {\rm (}possibly empty{\rm )}
on $R$ where the singularity becomes of Type 8.
\smallskip
\noindent

{\rm (b)} Suppose that $\pi=T_pM$. Then $j^2P_{\pi}$ is $\mathcal A^{(2)}$-equivalent to $(0,xy)$.

\end{thm}

\begin{proof} The proof is analogous to the Theorem \ref{theo:notInf} with the conditions on planes $\pi$ with $\mu\neq0$.
\end{proof}

\begin{thm} \label{theo:inf}
Let $R$ and $\Pi$ be as above and suppose that $p$ 
an inflection point of real type at the origin with $\mu=0$.

\noindent 
{\rm (a)} Suppose that $\dim (T_pR\cap \pi)=1$.

{\rm (a1)}  Let $S\subset \Pi$ be a surface. If $\pi\in S$ and away from a curve $C$, the singularity of the map-germ $P_{\pi}$ is of Type $4_{2}^-$ {\rm(}beaks{\rm)}. 

{\rm (a2)} For $\pi\in C$ and away from three planes $\pi_1, \pi_2, \pi_3$, 
the singularity of $P_{\pi}$ is of Type $11_{5}$ {\rm(}Gulls{\rm)}. For $\pi=\pi_1$ or $\pi=\pi_2$, the singularity of $P_{\pi}$ is of Type $11_{7}$.
%For each of the four planes the singularity of $P_{\pi}$ becomes of Type  $11_{9}$.
For $\pi=\pi_3$,  the singularity of $P_{\pi}$ is of \mbox{Type $12$} or worse.
%For one of the planes it is of Type $13$ and for the other it is of Type $15$.

%\smallskip
%\noindent
%{\rm (b)} Suppose that $\pi=T_pR$.\\
%\textcolor{red}{Then $P_{\pi}$ is $3$-$\mathcal K$-determined and is $\mathcal K$-equivalent to $(x^3-3xy^2,3x^2y-y^3)$.
%It is not $\mathcal A$-finitely determined.}
\end{thm}

\begin{proof}
(a) 
We can make changes of coordinates in the source so that $P_{\pi}(x,y)\sim_{\mathcal A} (y,g_{(\alpha,\lambda,0)}(x,y)).$
We have
$$
j^3g=\left( \bar{a}_{21}\alpha+\bar{a}_{30}\right)x^3+\bar{a}_{21}x^2y.
$$

If we denote by $g_{3i}$ the coefficient of $x^{3-i}y^i$ in $j^3g$, then the condition for $g$ to have a beaks singularity is 
$g_{30} \ne 0$ and $g_{31}^2-3g_{32}g_{30}=\bar{a}_{21}^2>0$. 
Therefore, when  $g_{30} \ne 0$, the singularity is always of type beaks.

The curve $C\subset \Pi$ where the singularity of $P_{\pi}$ is more degenerate than beaks is given by $g_{30} =0$, 
that is, $C: \, \bar{a}_{21}\alpha+\bar{a}_{30}=0.$ Then we get $\alpha=\bar{a}_{30}/\bar{a}_{21}$ on $C$.

For $\pi\in C$, we have $g_{31}=\bar{a}_{21}\ne 0$, 
so $j^3P_{\pi}\sim_{\mathcal A}(y,x^2y)$.

The singularity of $P_{\pi}$ is of Type $11_{2k+1}$ when the coefficient of $x^4$ in $g$ is not zero, that is, when
$
-\bar{a}_{21}\bar{a}_{30}\lambda+\bar{a}_{21}\bar{a}_{40}-\bar{a}_{30}\bar{a}_{31}\neq 0.
$
 When this is the case, if 
$$
 \begin{array}{l}
 \bar{a}_{21}^2\bar{a}_{30}\lambda^2+\bar{a}_{21}(\bar{a}_{21}^2\bar{b}_{30}-\bar{a}_{21}\bar{a}_{30}\bar{b}_{21}-2\bar{a}_{21}\bar{a}_{40}+3\bar{a}_{30}\bar{a}_{31})\lambda\\
\;\;\;\;\;\; +\bar{a}_{50}\bar{a}_{21}^2-\bar{a}_{41}\bar{a}_{30}\bar{a}_{21}-2\bar{a}_{21}\bar{a}_{31}\bar{a}_{40}+\bar{a}_{32}\bar{a}_{30}^2+2\bar{a}_{30}\bar{a}_{31}^2\neq0
 \end{array}
 $$
 means that away from at most two isolated point $P$ on $C$ (depends on discriminant the above equation in $\lambda$), the singularity is of Type $11_5$ (gulls).  
At each isolated point $P$, the singularity is of Type $11_7$ or $11_9$. %unless $\lambda$ is a root of a quadric polynomial. For one of these roots, the singularity becomes of Type $11_9$. 

At the isolated point on $C$ given by 
$\pi_3: -\bar{a}_{21}\bar{a}_{30}\lambda+\bar{a}_{21}\bar{a}_{40}-\bar{a}_{30}\bar{a}_{31}= 0$, the singularity is of Type 12,13 or 15.
%At one point it is of Type 13 and at the other it is of Type 15. 
\end{proof}

%We denote by $\mathcal R(2, 4)$ the set of germs of regular ruled surfaces  $(R, 0) \to(\mathbb R^4, 0)$ endowed with the Whitney topology.

% % % % % % % % % % % %
%\subsection{At a degenerate inflection point}

%%%%%%%%%%%%%%%%%%%%

\section{Projective equivalence}\label{SectProj}

The robust features presented in Section \ref{sec:singprojc} are affine invariant (because they depend
only on the contact of the surface with 2-planes) and are also projective invariant.  So, we will utilize projective transformations to simplify the coefficients of the Monge form of a ruled surface. Based on the findings presented in \cite{DK} we can use the projective transformation as follows. 

Let $0 \in \R^4 \subset \R P^4$ and $(x,y,z,w)$ a linear coordinate system of $\R^4$. 
A projective transformation on $\R P^4$ preserving the origin and the $xy$-plane
defines a diffeomorphism-germ $\Psi: (\R^4, 0) \to (\R^4,0)$ of the form 
with 16 coefficients 
$$ 
\Psi(x,y,z,w)=\left(\frac{q_1(x,y,z,w)}{p(x,y,z,w)},
\frac{q_2(x,y,z,w)}{p(x,y,z,w)},
\frac{q_3(x,y,z,w)}{p(x,y,z)},\frac{q_4(x,y,z,w)}{p(x,y,z,w)}\right)
$$
where
\begin{gather*}
q_1=q_{11} x + q_{12} y + q_{13} z+q_{14}w, \;\;\; q_2=q_{21} x + q_{22} y+q_{23}z+q_{24}w, \;\;\; q_3=q_{33}z+q_{34}w, \;\;\; \\q_4=q_{43}z+q_{44}w\nonumber  \;\;\;
p=1+p_1 x + p_2 y + p_3 z+p_4 w.\nonumber
\end{gather*}

For Monge forms $z=f^1(x,y)$, $w=f^2(x,y)$ and $z=g^1(x,y)$, $w=g^2(x,y)$
we denote $j^k(f^1,f^2) \sim j^k(g^1,g^2)$
if $k$-jets of these surfaces at the origin are projectively equivalent.
That is equivalent to that there is some projective transformation
$\Psi$ of the above form so that
$$
F(x,y,g^1(x,y),g^2(x,y))=o(k)
$$
where $F(x,y,z,w)=(q_3/p-f^1(q_1/p, \,q_2/p), q_4/p-f^2(q_1/p, \,q_2/p))$ 
and $o$ means Landau's symbol of function-germs of order greater than $k$ 
(in fact, $F(x,y,z,w)=0$ defines the image via $\Psi^{-1}$ of the graph $(z,w)-(f^1(x,y),f^2(x,y))=(0,0)$). 
Given $f^1(x,y)=\sum_{i+j\ge 2} a_{ij}x^iy^j$  and $f^2(x,y)=\sum_{i+j\ge 2} b_{ij}x^iy^j$
and an expected normal forms $g^1(x,y)$ and $g^2(x,y)$, 
our task is to find $\Psi$ satisfying the above equation in a practical way. 
%For instance, if $j^4f=xy+x^3+a_{40}x^4+\cdots + a_{04}y^4$, 
%then $j^4f \sim j^4g=xy+x^3$ ($k=4$). 
%To show this, 
%we first substitute into $F(x,y,g(x,y))$ 
%the data of $4$-jets of $f$ and $g$ together with unknown 10 coefficients, 
%then the equation yields a system of 
%polynomial equations of $u_1, \cdots, w_3$. 
%We solve it (indeed, it is enough to find one solution) 
%to obtain an explicit form of $\Psi$ sending $j^4f$ to $j^4g$; 
%each of unknowns is given by a rational function of 
%coefficients $a_{ij}\; (i+j=4)$.  

\begin{prop} At a parabolic point of the ruled surface $R$, the 4-jet of the Monge form of $R$ is projectively equivalent to   
\begin{equation}\label{NormalForm}
(x,y,x^2+\Gamma_{40} x^4+\Gamma_{31} x^3y, xy+\Theta_{40} x^4)
\end{equation}
\end{prop}
\begin{proof} Consider the Monge form as in (\ref{parabolicMonge}). Take the projective transformation 
$\Psi$ with
$q_1= \textstyle x + (\bar{b}_{21}-\bar{a}_{30})z-\bar{a}_{21}w, \;\; q_2=y-\bar{b}_{30} z, \;\; q_3=z, \;\;q_4=w,
\textstyle p=\textstyle1+(-\bar{a}_{30}+2\bar{b}_{21})x-\bar{a}_{21}y+(-\bar{a}_{21}\bar{b}_{30}-\bar{a}_{30}\bar{b}_{21}+\bar{b}_{31})z+(-\bar{a}_{21}\bar{b}_{21}+\bar{b}_{22})w$
to obtain the normal form (\ref{NormalForm}).
We can express the coefficients of (\ref{NormalForm})
in terms of the coefficients of (\ref{parabolicMonge}).
Especially,
we have
\begin{align*}
\left\{
\begin{array}{lll}
\Gamma_{31}=\bar{a}_{31}-\bar{a}_{21}\bar{a}_{30}-\bar{a}_{21}\bar{b}_{21}-\bar{b}_{22} \\
\Gamma_{40}=\bar{a}_{40}-\bar{a}_{30}^2+\bar{b}_{21}^2-\bar{b}_{31}\\
\Theta_{40} =\bar{b}_{40}-\bar{a}_{30}\bar{b}_{30}-\bar{b}_{21}\bar{b}_{30}\\
\end{array}
\right.
\end{align*}
\qed
\end{proof}

%\begin{rem}
%Note that, if we rewrite Theorem \ref{theo:notInf} with the 4-jet using the normal form (\ref{NormalForm}), Equation (\ref{ButterflyDirections}) is given simply by 
%\begin{equation*}%\label{simpleBDE}
%\Gamma_{31}\alpha^2+ \Gamma_{40}\alpha-\Theta_{40}=0.
%\end{equation*}
%\end{rem}
%From Definition \ref{ButterflyPoints}, a parabolic point $p$ of the ruled surface $R$ is a butterfly hyperbolic/parabolic/elliptic  if $\Gamma_{40}^2-4\Gamma_{31}\Theta_{40}>0/=0/<0$ at $p$.

\

\begin{cor}\label{5-jet} At a parabolic point of the ruled surface $R$ with the 4-jet of the Monge form is given by as in (\ref{NormalForm}), then the 5-jet of $R$ is projectively equivalent to   
\begin{equation}\label{NormalForm2}
(x^2+\Gamma_{40} x^4+\Gamma_{31} x^3y+\Gamma_{50}x^5+\Gamma_{41}x^4y+\Gamma_{32}x^3y^2, xy+\Theta_{40} x^4+\Theta_{50}x^5+\Theta_{41}x^4y)
\end{equation}
where $\Gamma_{31}\neq0$.
\end{cor}
\begin{proof} Consider the 4-jet Monge form as in (\ref{NormalForm}). So we can take the 5-jet of $R$ as 
$$
(x^2+\Gamma_{40} x^4+\Gamma_{31} x^3y+A_{50}x^5+A_{41}x^4y+A_{32}x^3y^2, xy+\Theta_{40} x^4+B_{50}x^5+B_{41}x^4y+B_{32}x^3y^2)
$$
Take the projective transformation 
$\Psi$ with
$q_1= -\frac{B_{32}}{\Gamma_{31}}z, \;\; q_2=y-\frac{B_{32}}{\Gamma_{31}}z, \;\; q_3=z, \;\;q_4=w,
\textstyle p=\textstyle1-\frac{2B_{32}}{\Gamma_{31}}x+(\frac{B_{32}}{\Gamma_{31}})^2z$
to obtain the normal form desired. Here
$\Gamma_{50}={A}_{50}$, $\Gamma_{41}={A}_{41}$, $\Gamma_{32}={A}_{32}$, 
$\Theta_{50}=\displaystyle\frac{\Gamma_{31}B_{50}+B_{32}B_{40}}{\Gamma_{31}}$ and 
$\Theta_{41} =\displaystyle\frac{\Gamma_{31}B_{41}+\Gamma_{40}B_{32}}{\Gamma_{31}}$.

\qed
\end{proof}

%\begin{prop} At an inflection point of the ruled surface $R$, the 3-jet of the Monge form of $R$ is projectively equivalent to   
%\begin{equation}\label{NormalForm}
%(\Gamma_{30} x^3+\Gamma_{21} x^3y, xy+\Theta_{40} x^4)
%\end{equation}
%\end{prop}
%\begin{proof} Consider the Monge form as in (\ref{parabolicMonge}). Take the projective transformation 
%$\Psi$ with
%$q_1= \textstyle x + (\bar{b}_{21}-\bar{a}_{30})z-\bar{a}_{21}w, \;\; q_2=y-\bar{b}_{30} z, \;\; q_3=z, \;\;q_4=w,
%\textstyle p=\textstyle1+(-\bar{a}_{30}+2\bar{b}_{21})x-\bar{a}_{21}y+(-\bar{a}_{21}\bar{b}_{30}-\bar{a}_{30}\bar{b}_{21}+\bar{b}_{31})z+(-\bar{a}_{21}\bar{b}_{21}+\bar{b}_{22})w$
%to obtain the normal form (\ref{NormalForm}).
%We can express the coefficients of (\ref{NormalForm})
%in terms of the coefficients of (\ref{parabolicMonge}).
%Especially,
%we have
%\begin{align*}
%\left\{
%\begin{array}{lll}
%\Gamma_{31}=\bar{a}_{31}-\bar{a}_{21}\bar{a}_{30}-\bar{a}_{21}\bar{b}_{21}-\bar{b}_{22} \\
%\Gamma_{40}=\bar{a}_{40}-\bar{a}_{30}^2+\bar{b}_{21}^2-\bar{b}_{31}\\
%\Theta_{40} =\bar{b}_{40}-\bar{a}_{30}\bar{b}_{30}-\bar{b}_{21}\bar{b}_{30}\\
%\end{array}
%\right.
%\end{align*}
%\qed
%\end{proof}

%%%%%%%%%%%%%%%%%%%%%%%%%%%%%
\section{Binary Differential Equation}\label{secBDE}

According to Equation (\ref{ButterflyDirections}) in Theorem \ref{theo:notInf}, 
at parabolic points  of the ruled surface $R$, 
there are at most two (butterfly) directions in the tangent plane which determine 
two planes along which the parallel projection has a singularity of Type 6 (butterfly).
In Section \ref{sec:singprojc}, we chose local coordinates at a point $p_0=(0,t_{0})$ and the ruled surface is locally the image of the $\varphi(x,y)=(x,y,f^1(x,y),f^2(x,y))$. 
Then Equation (\ref{ButterflyDirections}) gives 
the Type 6 planes at $p_0$. For $q = \varphi(p_0)$ in a neighbourhood of $p_0$, we can make a translation and
an affine transformation and write at $q$ in the form $\varphi_p(u,v) = (u,v, f^1_p(u,v),f^2_p(u,v))$ with $f^1_p(u,v)$ and $f^2_p(u,v)$ given by 
$$
f^1_p(u,v)=\frac12(\underline{a}_{20}u^2+2\underline{a}_{11}uv)+\frac16(\underline{a}_{30}u^3+3\underline{a}_{21}u^2v)+\frac{1}{24}(\underline{a}_{40}u^4+4\underline{a}_{31}u^3v+6\underline{a}_{22}u^2v^2)+\cdots
$$
$$
f^2_p(u,v)=\frac12(\underline{b}_{20}u^2+2\underline{b}_{11}uv)+\frac16(\underline{a}_{30}u^3+3\underline{b}_{21}u^2v)+\frac{1}{24}(\underline{b}_{40}u^4+4\underline{b}_{31}u^3v+6\underline{b}_{22}u^2v^2)+\cdots
$$
%the Taylor expansion at $(u,v)$ with coefficients $\underline{a}_{ij}(u,v)$ and $\underline{b}_{ij}(u,v)$, respectively. 
where $\displaystyle\underline{a}_{ij}(u,v)=\frac{\partial f^1}{\partial u^i\partial v^j}(u,v)$ and $\displaystyle\underline{b}_{ij}(u,v)=\frac{\partial f^2}{\partial u^i\partial v^j}(u,v),$ with $\displaystyle\underline{a}_{ij}(0,0)=\frac{\partial f^1}{\partial u^i\partial v^j}(0,0)=i!\cdot j!\cdot a_{ij}$
and $\displaystyle\underline{b}_{ij}(0,0)=\frac{\partial f^2}{\partial u^i\partial v^j}(0,0)=i!\cdot j!\cdot b_{ij}$. 
%We write,  ${\bf a}_k(x,y)=\underline{a}_k(x,y)+i\underline{b}_k(x,y)=\frac{1}{k!}f^{(k)}(z)$, $\underline{a}_k(0,0)=a_k$ and  $\underline{b}_k(0,0)=b_k$.
Redoing the calculations for the butterfly directions at $q$ %and without setting $b_{11}=1$ 
and considering the coefficients of the above Monge form we found that these are given by the binary differential equation (BDE) 
\begin{equation}\label{BDE}
A(u,v)dv^2+B(u,v)dudv+C(u,v)du^2=0
\end{equation}
with 
$$
\begin{array}{rcl}
A(u,v)&=&6(\underline{a}_{11}\underline{b}_{20}-\underline{a}_{20}\underline{b}_{11})\Big(4(\underline{a}_{11}\underline{b}_{31}-\underline{a}_{31}\underline{b}_{11})+3(\underline{a}_{20}\underline{b}_{22}-\underline{a}_{22}\underline{b}_{20})\Big)\\
&& -12(\underline{a}_{11}\underline{b}_{21}-\underline{a}_{21}\underline{b}_{11})\Big(2(\underline{a}_{11}\underline{b}_{30}-a_{30}\underline{b}_{11})+3(\underline{a}_{21}\underline{b}_{20}-\underline{a}_{20}\underline{b}_{21})\Big)\\

%-(a_{20}a_{31}-a_{21}a_{30})b_{11}^3+(a_{11}a_{20}b_{31}-a_{11}a_{21}b_{30}-a_{11}a_{30}b_{21}+a_{11}a_{31}b_{20}\\
%&&+a_{20}^2b_{22}+a_{20}a_{21}b_{21}-a_{21}^2b_{20})b_{11}^2-a_{11}(a_{11}b_{20}b_{31}-a_{11}b_{21}b_{30}+2a_{20}b_{20}b_{22}\\
%&&+a_{20}b_{21}^2-a_{21}b_{20}b_{21})b_{11}+a_{11}^2b_{20}^2b_{22}\\
B(u,v)&=&2\Big(3(\underline{a}_{20}\underline{b}_{21}-\underline{a}_{21}\underline{b}_{20})\Big)^2-2\Big(2(\underline{a}_{11}\underline{b}_{30}-\underline{a}_{30}\underline{b}_{11})\Big)^2\\
&& +6(\underline{a}_{11}\underline{b}_{20}-\underline{a}_{20}\underline{b}_{11})\Big((\underline{a}_{11}\underline{b}_{40}-\underline{a}_{40}\underline{b}_{11})+2(\underline{a}_{20}\underline{b}_{31}-\underline{a}_{31}\underline{b}_{20})\Big)\\

%-(a_{20}a_{40}-a_{30}^2)b_{11}^3+(a_{11}a_{20}b_{40}-2a_{11}a_{30}b_{30}+a_{11}a_{40}b_{20}+a_{20}^2b_{31}\\
%&&-a_{20}a_{31}b_{20})b_{11}^2-(a_{11}^2b_{20}b_{40}-a_{11}^2b_{30}^2+a_{11}a_{20}b_{20}b_{31}-a_{11}a_{31}b_{20}^2+a_{20}^2b_{21}^2\\
%&&-2a_{20}a_{21}b_{20}b_{21}+a_{21}^2b_{20}^2)b_{11}\\
C(u,v)&=&3(\underline{a}_{20}\underline{b}_{40}-\underline{a}_{40}\underline{b}_{20})(\underline{a}_{11}\underline{b}_{20}-\underline{a}_{20}\underline{b}_{11})\\
&&-2(\underline{a}_{20}\underline{b}_{30}-\underline{a}_{30}\underline{b}_{20})\Big(2(\underline{a}_{11}\underline{b}_{30}-\underline{a}_{30}\underline{b}_{11})+3(\underline{a}_{21}\underline{b}_{20}-\underline{a}_{20}\underline{b}_{21})\Big).\\

%-(a_{20}a_{40}-a_{30}^2)b_{11}^3+(a_{11}a_{20}b_{40}-2a_{11}a_{30}b_{30}+a_{11}a_{40}b_{20}+a_{20}^2b_{31}\\
%&& -a_{20}a_{31}b_{20})b_{11}^2-(a_{11}^2b_{20}b_{40}-a_{11}^2b_{30}^2+a_{11}a_{20}b_{20}b_{31}-a_{11}a_{31}b_{20}^2+a_{20}^2b_{21}^2\\
%&&-2a_{20}a_{21}b_{20}b_{21}+a_{21}^2b_{20}^2)b_{11}.
\end{array}
$$

In Equation (\ref{BDE}), the direction 
$(1,\alpha)=(1,dv/du)$ in the tangent plane of $R$ at $(u,v)$ is represented by $(du,dv)$ to include 
the direction $(0,1)$. %An immediate consequence is the following.

\

%Now we are interesting in restrict our study in BDE in parabolic points and inflections points. 
To study the configurations butterfly lines we need some results on BDEs. See, for example, \cite{pairs} for a survey article. Consider the BDE 
\begin{equation*}
\Omega(u, v, {\bf p}) =A(u, v)dv^2 + 2B(u, v)dudv + C(u, v)du^2 = 0,
\end{equation*}
where ${\bf p}=dv/du$.   The set $\Delta=0$, where $\Delta = B^2 -AC$, is the {\it discriminant curve} of BDE at which the integral curves generically have  cusps. The
BDE defines two directions in the plane when $\Delta> 0$.  
These directions lifts to a single valued field $\xi$ on $\mathcal M=\Omega^{-1}(0)$. 
A suitable lifted field $\xi$ %(see \cite{davbook}) 
 is given by $ \Omega_{\bf p}\partial u + {\bf p}\Omega_{\bf p}\partial v-(\Omega_u + {\bf p}\Omega v)\partial {\bf p}$.

One can separate BDE  into two types.
The first case occurs when the functions $A, B, C$ do not all vanish at the origin.
Then the BDE is an implicit differential equation (IDE).
The second case is when all the coefficients of BDE vanish at the origin.
Stable topological models of the BDEs belong to the first case; they arise when the discriminant is smooth (or empty).
If the unique direction at a point of the discriminant is transverse to it,
then the BDE is smoothly equivalent to
$dv^2+udu^2=0$. %(\cite{cibrario,dara}).
If the unique direction is tangent to the discriminant
 and given additional conditions %(see \cite{duality} and \cite{davbook}) 
 the BDE is smoothly equivalent to
$dv^2+(-v+\lambda u^2)du^2=0$
with $\lambda\neq 0,\frac{1}{16}$. %(\cite{davbook}); 
The corresponding point in the plane is called a \emph{folded singularity} 
-- more precisely, 
a \emph{folded saddle} if $\lambda<0$, 
a \emph{folded node} if $0<\lambda<\frac{1}{16}$, 
and a \emph{folded focus} if $\frac{1}{16}<\lambda$, (see Figure \ref{Folded}). %and \cite{davbook}). 

%A solution curve has an {\it inflection point} at the projection of points when $\Omega=\Omega_x+p\Omega_y=0$. There is a smooth curve of such points tangent to the discriminant curve at folded singularities (\cite{duality}).
	
\begin{figure}[htp]
\begin{center}
\includegraphics[width=4.5in, height=1.5cm]{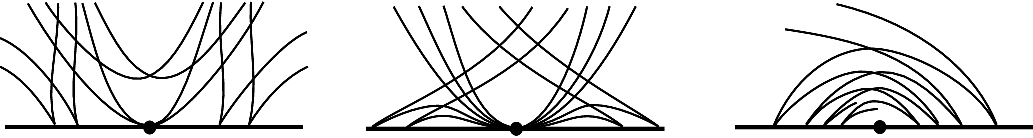}
\caption{A folded saddle (left), node (center) and focus (right).}
\label{Folded}
\end{center}
\end{figure}

%For surfaces in $\mathbb R^4$ away from inflection points, the asymptotic curves are generically a family of cusps at ordinary parabolic points and have a folded singularity at a $P_3(c)$-point of the projection $P_{\bf u}$ (see \cite{r4surf} and Figure \ref{Folded}).

\begin{prop} 
The BDE  (\ref{BDE}) of the butterfly directions is well defined at  inflection
points of real type on the ruled surface. For a generic ruled surface, when $p$ is an inflection point of real type on $R$ and
%
%\begin{itemize}
%\item[(i)] 
 $2(\underline{a}_{11}\underline{b}_{30}-\underline{a}_{30}\underline{b}_{11})+3(\underline{a}_{21}\underline{b}_{20}-\underline{a}_{20}\underline{b}_{21})\neq0$, its discriminant curve (butterfly parabolic curve) is transversal to the inflection curve at $p$.

%\item[(ii)] \textcolor{red}{if $2(\underline{a}_{11}\underline{b}_{30}-\underline{a}_{30}\underline{b}_{11})+3(\underline{a}_{21}\underline{b}_{20}-\underline{a}_{20}\underline{b}_{21})=0$, the singularity of BDE occurs at an isolated point $p$ on $R$.} %such that 
%\begin{eqnarray*}
%2(\underline{a}_{11}\underline{b}_{30}-\underline{a}_{30}\underline{b}_{11})+3(\underline{a}_{21}\underline{b}_{20}-\underline{a}_{20}\underline{b}_{21})=0%&=& (n_1q_2-n_2q_1)\Big((d_3n_1-\ell_3q_1)+(d_2n_2-\ell_2q_2)\Big)\\
%&+&(d_2n_1-\ell_2q_1)\Big(m_2(d_2n_1-\ell_2q_1)+(n_3q_1-n_1q_3)\Big)=0
%\end{eqnarray*}
%\end{itemize}
\end{prop}

\begin{proof}
It is clear that the coefficients of the BDE (\ref{BDE}) are well defined when $b_{20}a_{11}-b_{11}a_{20}=0$,  i.e., at an inflection point of real type. 
%
%$(i)$ 
In fact, at an inflection point of real type, when  $2(\underline{a}_{11}\underline{b}_{30}-\underline{a}_{30}\underline{b}_{11})+3(\underline{a}_{21}\underline{b}_{20}-\underline{a}_{20}\underline{b}_{21})\neq0$, the 1-jet of the coefficients of the BDE  around $p$ are given by 
%$$
%-6(a_{11}b_{21}-a_{21}b_{11})dv^2+\left(3(a_{20}b_{21}-a_{21}b_{20})-2(a_{11}b_{30}-a_{30}b_{11})\right)dudv-(a_{20}b_{30}-a_{30}b_{20})du^2=0.
%$$
$$
\begin{array}{rcl}
A(u,v)&=& -6(\underline{a}_{11}\underline{b}_{21}-\underline{a}_{21}\underline{b}_{11})=12(b_{11}a_{21}-a_{11}b_{21})+36(-a_{11}b_{31}-a_{31}b_{11})u\\
B(u,v)&=&3(\underline{a}_{20}\underline{b}_{21}-\underline{a}_{21}\underline{b}_{20})-2(\underline{a}_{11}\underline{b}_{30}-\underline{a}_{30}\underline{b}_{11})\\
&=&12((a_{30}b_{11}-a_{11}b_{30})+(a_{21}b_{20}-a_{20}b_{21}))+12(a_{31}b_{11}-a_{11}b_{31})v+\\
&&+12(4(a_{40}b_{11}-a_{11}b_{40})+3(a_{31}b_{20}-a_{20}b_{31})+(a_{21}b_{30}-a_{30}b_{21}))u\\
C(u,v)&=&-(\underline{a}_{20}\underline{b}_{30}-\underline{a}_{30}\underline{b}_{20})=12(a_{20}b_{30}-b_{20}a_{30})+4(a_{20}b_{40}-a_{40}b_{20})u\\
&&12((a_{20}b_{31}-a_{31}b_{20})+(a_{21}b_{30}-a_{30}b_{21})v.\\
\end{array}
$$
The discriminant curve is $\Delta(u,v)=B^2-4AC=0$. If $p$ is also a point of discriminant curve, then   
$ ((a_{30}b_{11}-a_{11}b_{30})+(a_{21}b_{20}-a_{20}b_{21}))^2- (b_{11}a_{21}-a_{11}b_{21})(a_{20}b_{30}-b_{20}a_{30})=0$.
%$$
%\left(3(a_{20}b_{21}-a_{21}b_{20})-2(a_{11}b_{30}-a_{30}b_{11})\right)^2-24(a_{11}b_{21}-a_{21}b_{11})(a_{20}b_{30}-a_{30}b_{20})=0.
%$$
Since, the 1-jet of the each coefficient of the BDE depends on the 4-jet of $R$, then the discriminant curve at $p$ is parametrised by $\Delta(u,v)=C_1u+C_2v+\cdots=0$ where $C_1$ and $C_2$ depends on the 4-jet of surface. By Corollary \ref{InfleCurve} the 1-jet of the inflection curve depends on the 3-jet of the surface, so the result follows.

%$(ii)$ In opposite case, we obtain the singularities of BDE when $(\underline{a}_{11}\underline{b}_{20}-\underline{a}_{20}\underline{b}_{11})=0$ and $2(\underline{a}_{11}\underline{b}_{30}-\underline{a}_{30}\underline{b}_{11})+3(\underline{a}_{21}\underline{b}_{20}-\underline{a}_{20}\underline{b}_{21})=0$.

 \qed
%At inflection point of flat type, if $(a_{20}b_{21}-a_{21}b_{20})\neq0$ the BDE is an ordinary differential equation (ODE) determined by
%$$
%(a_{20}b_{21}-a_{21}b_{20})dv-(a_{20}b_{30}-a_{30}b_{20})du=0.
%$$
%Then the direction $\displaystyle (1,\alpha)=\left(1,\frac{a_{20}b_{30}-a_{30}b_{20}}{a_{20}b_{21}-a_{21}b_{20}}\right)$. \qed
\end{proof}

Now we are interested in restricting our study in BDE in parabolic points away from inflection points of real type. 

Recall that on a surface in 4-space, it is known that there are two (resp. or one, or none) asymptotic directions determined by BDE at hyperbolic (resp. parabolic, or elliptic) points. At parabolic point in a smooth ruled surface in 4-space, it is possible to identify 2, 1, or 0 butterfly directions, depending on the sign of the discriminant of Equation (\ref{BDE}). In this context, 
 in an analogy to the behavior of the asymptotic directions and curves on surfaces in $4$-space,
%considering only the quantity of directions that may exist at each parabolic point $p$, 
we can call the point $p$ by {\it butterfly hyperbolic/parabolic/elliptic} when the discriminant of the BDE (\ref{BDE}) $\Delta>0/=0/<0$ at $p$.

Consider $p$ is a parabolic point at $R$. By Corollary \ref{5-jet} we can consider the 5-jet of the Monge form of $R$ as in (\ref{NormalForm2}), then the 1-jets of the coefficients of the BDE $\Omega(u,v,{\bf p})$ at $(0,0,0)$ are given by
\begin{equation}\label{BDEfinal}
\begin{array}{l}
A(u,v)=\Gamma_{31}+2 \Gamma_{32} v+4 \Gamma_{41} u\\
B(u,v)=-2\Gamma_{40}+(-10\Gamma_{50}+8\Theta_{41})u-2\Gamma_{41}v\\
C(u,v)= -\Theta_{40}-5\Theta_{50}u-\Theta_{41}v
%+(10\Theta_{40}\Gamma_{40}-32\Theta_{40}^2-15\Theta_{60})u^2\\
%&&-(2\Theta_{40}\Gamma_{31}+5\Theta_{51})uv-\Theta_{42}v^2\Big).
 \end{array}
\end{equation}
As an immediate consequence, if $p$ is the origin, we get that $\Delta(p)=4(\Gamma_{40}^2+\Gamma_{31}\Theta_{40})$. Then we obtain the following.

\begin{cor} \label{ButterflyPoints}
Let $p$ be a parabolic point in a smooth ruled surface $R$. Consider the 4-jet of $R$ given in Monge form as in (\ref{NormalForm}), then $p$ is a {\normalfont butterfly hyperbolic/parabolic/elliptic} if $\Gamma_{40}^2+\Gamma_{31}\Theta_{40}>0/=0/<0$ at $p$.
\end{cor}

%We study more carefully these point in Section \ref{secBDE}.

\begin{thm}
The solutions of the BDE (\ref{BDEfinal}) of the butterfly directions form a 
pair of transverse foliations on surface $R$ in butterfly hyperbolic region. Furthermore: 
\begin{itemize}
\item[(i)] The butterfly parabolic curve is regular and the unique direction defined by  is transverse to the curve; the BDE is smoothly equivalent to
$$
dv^2+udu^2=0.
$$
%if $(5\Gamma_{31}^2\Theta_{50}+10\Gamma_{31}\Gamma_{40}\Gamma_{50}-8\Gamma_{31}\Gamma_{40}\Theta_{41}-4\Gamma_{40}^2\Gamma_{41})\neq0$.
\item[(ii)] The butterfly parabolic curve is regular and the unique direction defined by is tangent to the curve; the BDE is smoothly equivalent to
$$
dv^2+(-v+\lambda u^2)du^2=0
$$
where $\lambda\neq0,\frac{1}{16}$ and it depends on the coefficients of order $6$.
\end{itemize}
\end{thm}

\begin{proof}
The BDE (\ref{BDEfinal}) defines a pair of transverse directions at each point where its discriminant $\Delta$ is positive, i.e., at butterfly hyperbolic points. At butterfly parabolic point, $\Gamma_{40}^2-\Gamma_{31}\Theta_{40}=0$. So if $\Gamma_{31}\neq0$, the butterfly parabolic curve is a regular curve with 1-jet given by 
$$
\frac{(5\Gamma_{31}^2\Theta_{50}+10\Gamma_{31}\Gamma_{40}\Gamma_{50}-8\Gamma_{31}\Gamma_{40}\Theta_{41}-4\Gamma_{40}^2\Gamma_{41})u}{\Gamma_{31}}+\frac{(\Gamma_{31}^2\Theta_{41}+2\Gamma_{31}\Gamma_{40}\Gamma_{41}-2\Gamma_{32}\Gamma_{40}^2)v}{\Gamma_{31}}.
$$

$(i)$ The 1-jets of the coefficients of the BDE (\ref{BDEfinal}) are given,
after linear changes, by
$$
\begin{array}{rcl}
A(u,v)&=&\displaystyle\frac{\Gamma_{40}^2}{\Gamma_{31}}-5\Theta_{50}u+(-5\Gamma_{31}\Theta_{50}-\Gamma_{40}\Theta_{41})v\\
B(u,v)&=&(-10\Gamma_{31}\Theta_{50}-10\Gamma_{40}\Gamma_{50}+8\Gamma_{40}\Theta_{41})u\\
&& +(-10\Gamma_{31}^2\Theta_{50}-10\Gamma_{31}\Gamma_{40}\Gamma_{50}+6\Gamma_{31}\Gamma_{40}\Theta_{41}-2\Gamma_{40}^2\Gamma_{41})v\\
C(u,v)&=&(5\Gamma_{31}^2\Theta_{50}+10\Gamma_{31}\Gamma_{40}\Gamma_{50}-8\Gamma_{31}\Gamma_{40}\Theta_{41}-4\Gamma_{40}^2\Gamma_{41})u\\
&& +(-5\Gamma_{31}^3\Theta_{50}-10\Gamma_{31}^2\Gamma_{40}\Gamma_{50}+7\Gamma_{31}^2\Gamma_{40}\Theta_{41}+2\Gamma_{31}\Gamma_{40}^2\Gamma_{41}+2\Gamma_{32}\Gamma_{40}^3)v.
\end{array}
$$
Using \cite{BruceTari2} the BDE is equivalent to $dv^2+udu^2=0$ if and only if $\xi=(5\Gamma_{31}^2\Theta_{50}+10\Gamma_{31}\Gamma_{40}\Gamma_{50}-8\Gamma_{31}\Gamma_{40}\Theta_{41}-4\Gamma_{40}^2\Gamma_{41})\neq0$.

$(ii)$ When $\xi = 0$ but $(\Gamma_{31}^2\Theta_{41}+2\Gamma_{31}\Gamma_{40}\Gamma_{41}-2\Gamma_{32}\Gamma_{40}^2)\neq 0$,  butterfly parabolic curve is still regular. We can reduce the 2-jet, and hence the BDE, to $
dv^2+(-v+\lambda u^2)du^2=0$ if $\lambda\neq0,\frac{1}{16}$, where it depends on the coefficients of order $6$.
\qed
\end{proof}

\

{\bf Acknowledgements:} The author would like to thank F. Tari and Y. Kabata for their comments.

%%%%%%%%%%%%%%%%%%%%%%%%%%%%%%%%%%%%%%%%%%%

\end{document}